\documentclass[11pt,reqno]{amsart}
\usepackage{mathrsfs}
\usepackage{amsmath}
\usepackage{amsthm}
\usepackage{amssymb, amsmath}
\usepackage{graphicx}
\usepackage{color}
\usepackage{enumerate}
\makeatletter
\@namedef{subjclassname@2020}{\textup{2020} Mathematics Subject Classification}
\makeatother

\usepackage[numbers,sort&compress]{natbib}
\usepackage{cases}
\numberwithin{equation}{section}

\newtheorem{theorem}{Theorem}[section]
\newtheorem{proposition}{Proposition}[section]
\newtheorem{lemma}{Lemma}[section]
\newtheorem{remark}{Remark}[section]
\newtheorem{corollary}{Corollary}[section]
\newtheorem{definition}{Definition}[section]

\def\n{n}

\begin{document}
\author[D. Tang]{De Tang}\address{De Tang\newline\indent School of Mathematics (Zhuhai) \newline\indent Sun Yat-sen University\newline\indent Zhuhai,
519082, Guangdong,P.R. China}
\email{tangd7@mail.sysu.edu.cn}

\author[Z.-A. Wang]{Zhi-an Wang\textsuperscript{$ \ast $}}\thanks{$^{\ast}$Corresponding author.}\address{Zhi-an Wang\newline\indent Department of Applied Mathematics\newline\indent Hong Kong Polytechnic University\newline\indent Hung Hom,
Kowloon, Hong Kong,  P. R. China}\email{mawza@polyu.edu.hk}

\title[Predator-prey model with density-dependent dispersal]{\bf Coexistence of heterogenous predator-prey systems with density-dependent dispersal}

\begin{abstract}
This paper is concerned with existence, non-existence and uniqueness of positive (coexistence) steady states to a predator-prey system with density-dependent dispersal. To overcome the analytical obstacle caused by the cross-diffusion structure embedded in the density-dependent dispersal, we use a variable transformation to convert the problem into an elliptic system without cross-diffusion structure. The transformed system and pre-transformed system are equivalent in terms of the existence or non-existence of positive solutions. Then we employ the index theory alongside the method of the principle eigenvalue to give a nearly complete classification for the existence and non-existence of positive solutions. Furthermore we show the uniqueness of positive solutions and characterize the asymptotic profile of solutions for small or large diffusion rates of species. Our results pinpoint the positive role of density-dependent dispersal on the population dynamics for the first time by showing that the density-dependent dispersal is a beneficial strategy promoting the coexistence of species in the predator-prey system by increasing the chance of predator's survival.

\end{abstract}


\maketitle

\section{Introduction}

Dispersal, an ecological process involving the movement of individual/multiple species, is one of the main determinants shaping the structure of ecological communities and maintaining the biodiversity \cite{hiltunen2013relative, holt1977predation,dieckmann1999evolutionary}. The causes and consequences as well as the selection and evolution of dispersal strategies have been central questions in ecology extensively investigated in the literature (cf. \cite{shurin2001effects, ron2018dispersal, zhang2017carrying}). A variety of mathematical models have been constructed to explore the effects of dispersal strategies on the population dynamics and to predict their biological consequences (cf. \cite{Cosner-review,cantrell2006movement,CCL-MBE, lou2008some, CantrellCosner, jungel2010diffusive, okubo2001diffusion}) where most of existing theoretical studies are focused on the random dispersal. However,  biological species will more likely employ the non-random dispersal strategy to optimize their ecological fitness in changing environments such as local population size, resource competition, habitat quality/size, inbreeding avoidance, crowding effect and so on. Among various possible non-random dispersal strategies, the density-dependent dispersal (meaning that the dispersal rate of one species depends on the densities of others) has been a major topic for discussion in the biological literature (cf. \cite{travis1999evolution, matthysen2005density,smith2008effects,maag2018density}). 
A prototype of two interacting species models with density-dependent dispersal generally reads as
\begin{equation}\label{m}
\begin{cases}
u_t=\mu_1 \Delta u+f(x,u,v), &\mbox{in}\;\Omega\times (0,\infty),\\
v_t=\mu_2\Delta (d(u)v)+g(x,u,v), &\mbox{in}\;\Omega\times (0,\infty),\\
\end{cases}
\end{equation}
where $u(x,t)$ and $v(x,t)$ denote the population densities of two interacting species at location $x \in \Omega$ and at time $t>0$ in a bounded habitat $\Omega\subset \mathbb{R}^{\color{black}N} (N\geq 2)$.  $\Delta=\sum_{i=1}^{{\color{black}n}} \frac{\partial^2}{\partial x_i^2}$ is the usual Laplace operator. $\mu_1$ and $\mu_2$ are positive constants accounting for the diffusion rates of the two interacting species.
The functions $f$ and $g$ describe the intra-specific and inter-specific interactions between species in a possibly heterogenous environment.
The term $\Delta (d(u)v)$ entails that the dispersal of species $v$ depends on the density of species $u$ via the dispersal rate function $d(u)$. The dispersal strategy of species $v$ is said to be random if $d(u)$ is constant, while to be density-dependent if $d(u)$ is non-constant.
Endowing $f(x,u,v)$ and $g(x,u,v)$ with different forms, model \eqref{m} may include many well-known mathematical biology models with density-dependent dispersal, such as the Keller-Segel model \cite{KS71-traveling,Keller-Segel-JTB-1971} describing chemotaxis, density-suppressed motility model \cite{Fu-etal-PRL-2012} describing the bacterial strip pattern formation driven by self-trapping mechanism, and prey-taxis system \cite{Kareiva, tyutyunov2017prey}, starvation-driven diffusion \cite{Chunetc, Cho-Kim-2013}. Moreover model \eqref{m} with competitive dynamics can be regarded as a special case of Shigesada-Kawasaki-Teramoto (STK) competition model originally proposed in  \cite{shigesada1979spatial} (see also \cite{louNi1996diffusion}). Most of existing theoretical studies of \eqref{m} have been focused on the random dispersal for which  a large number of results have been developed on the steady state problem of \eqref{m}, for example see \cite{Dancer1984, Dancer,du1997some, du2007allee, Li-Wu-Dong-CVPDE, Li-Wu-Dong-Nonlinearity, nakashima1996positive, yi2009bifurcation}) for predator-prey systems  and \cite{dockery1998evolution, Lam-Memo,LamNi, lou2006effects, he2016global}) for competition systems;  see also \cite{CantrellCosner, Cosner-review} and references therein.  In recent years mathematical models with density-dependent dispersal have increasingly  received attentions.
Among other things, this paper is concerned with the following predator-prey system with density-dependent dispersal proposed in \cite{Kareiva}
\begin{equation}\label{KO}
\begin{cases}
u_{t}=\epsilon\Delta u+u(m(x)-u)-vF(u), &\mbox{in}\;\Omega\times \mathbb{R}^+,\\
v_{t}=\mu \Delta(d(u) v)+\alpha vF(u)-\theta v, &\mbox{in}\;\Omega\times \mathbb{R}^+,\\
\nabla u \cdot \n=\nabla v \cdot \n=0,&\mbox{on}\;\partial\Omega\times \mathbb{R}^+,\\
\end{cases}
\end{equation}
where $u$ denotes the prey density and $v$ the predator density; $\epsilon>0$ and $\mu>0$ account for the diffusion rates of the prey and predator, respectively; $F(u)$ denotes the functional response function and $\alpha>0$ is the conversion rate, $\theta>0$ is the predator's motility rate.  $\n$ denotes the outward unit normal vector of $\partial \Omega$ and  Neumann boundary conditions are prescribed to warrant that no individual crosses the habitat boundary. The system \eqref{KO} is a special form of prey-taxis models proposed in \cite{Kareiva} to describe the non-random foraging behavior of predators, where the dispersal rate function $d(u)$ satisfies the property $d'(u)<0$ complying with the field observation that the predator will reduce its random motility in the area of higher density of prey.

When $d(u)$ is constant, model \eqref{KO} becomes the classical diffusive predator-prey systems extensively studied in the literature, such as the steady-state problem (cf. \cite{Dancer1984, Dancer,du1997some, du2007allee, Li-Wu-Dong-CVPDE, Li-Wu-Dong-Nonlinearity, nakashima1996positive, yi2009bifurcation}) and traveling wave problem (cf. \cite{Huang-tws1, Huang-tws2}), just to mention some.
In contrast the available results for non-constant $d(u)$ are much less. The global boundedness of classical solutions and stability of constant steady states of \eqref{KO} with constant $m(x)$ was first established in \cite{jin2021global}. When $d(u)$ is a special form of piecewise decreasing function, the existence of non-constant steady state solutions of  \eqref{KO} with non-constant $m(x)$ for large $\epsilon>0$  was obtained in \cite{choi2023predator} under certain conditions and effect of predator satisfaction on predator's survival was examined. When $F(u)$ is replaced by a Leslie-Gower type functional response function, the global boundedness of solutions and global stability of constant positive solutions were recently obtained in \cite{mi2023global} for constant $m(x)$. From application point of view, an important question is how the density-dependent dispersal rate function $d(u)$ plays roles in the population dynamics, which, however, has not been explored in the above mentioned works.

The goal of this paper is to explore the steady state problem of \eqref{KO} with non-constant $d(u)$ and $m(x)$ to find conditions under which positive solutions exists, by which we pinpoint the role of density-dependent dispersal on the species coexistence. This is not only a biologically relevant  question (since coexistence of species is the center question concerned in ecology), but also an interesting mathematical question due to the inherent cross-diffusion structure in the model which makes many conventional methods fail to use.
The steady state problem of \eqref{KO}  reads as
\begin{equation}\label{2022032801}
\begin{cases}
\epsilon\Delta u+u(m(x)-u)-vF(u)=0, &\mbox{in}\;\Omega,\\
\mu\Delta (d(u)v)+\alpha vF(u)-\theta v=0, &\mbox{in}\;\Omega,\\
\nabla u \cdot \n=\nabla v \cdot \n=0,&\mbox{on}\;\partial\Omega.
\end{cases}
\end{equation}
Throughout the paper, we make the following basic assumptions on $m(x), F(u)$ and $d(u)$:
\begin{enumerate}
\item[($H_1$)]\quad $m\in C(\bar{\Omega})$, $\int_\Omega m(x)\mathrm{d}x>0$ and $m$ is not constant;\vspace{0.1cm}
\item[($H_2$)]\quad  $F(0)=0$ and $F'(u)>0$ for all $u\in[0,\infty)$;\vspace{0.1cm}
\item[($H_3$)]\quad $d(u)\in C^2([0,\infty))$, $d(u)>0$ and $d'(u)\leq, \not \equiv 0$ on $[0,\infty)$.
\end{enumerate}
The assumption ($H_1$) indicates that the resource $m(x)$ could be beneficial or harmful but the total mass of the resource is advantageous. The assumption ($H_2$) gives some basic property of the functional response function which can be fulfilled by a large class of function like Holling type I, II and III. The assumption ($H_3$) indicates that the random diffusion of the predator decreases with respect to the prey density, which is a biological postulation as in \cite{Kareiva}.
Though system \eqref{2022032801} has a cross-diffusion structure, fortunately we can circumvent this obstacle by
invoking a change of variable
\begin{equation}\label{trans}
w:=d(u)v
\end{equation}
 which reformulates   \eqref{2022032801} to the following elliptic problem without cross-diffusion
\begin{equation}\label{KO2}
\begin{cases}
\epsilon\Delta u+u(m(x)-u)-\frac{F(u)}{d(u)}w=0, &\mbox{in}\;\Omega,\\
\mu\Delta w+\frac{\alpha F(u)-\theta}{d(u)}w=0, &\mbox{in}\;\Omega,\\
\nabla u \cdot \n=\nabla w \cdot \n=0,&\mbox{on}\;\partial\Omega.
\end{cases}
\end{equation}
The reformulated problem \eqref{KO2} has conventional random diffusions only and many well-developed methods are potentially applicable.  However the reaction terms in \eqref{KO2} become more complicated under the transformation \eqref{trans} and existing methods and results for the predator-prey systems can not be applied directly. For example, the function $F(u)$ in \eqref{2022032801} is monotonic but the function $\frac{F(u)}{d(u)}$ in the transformed system \eqref{KO2} is no longer monotonic, which makes the analysis more difficult. The main goal of this paper is to find the existence conditions for the positive solutions of \eqref{2022032801} and hence pinpoint the effects of density-dependent dispersal on the coexistence of species. Noting that the existence/non-existence of positive solutions of \eqref{2022032801} is equivalent to that of \eqref{KO2} via the transformation \eqref{trans}, in what follows we shall focus on the transformed system \eqref{KO2} and fully exploit its structure alongside the delicate analysis to show that the parameter regimes of species coexistence is broadened by the density-dependent dispersal. This implies that the density-dependent dispersal is an advantageous strategy of increasing the biodiversity in a heterogenous landscape.

The main results of this paper consist of two parts. The first part is to find conditions for the existence and non-existence of positive (coexistence) solutions of \eqref{KO2}, which are given in Theorem \ref{thm01} by which we are able to  pinpoint the positive role of density-dependent dispersal in promoting the species coexistence. The second part is to further explore the uniqueness and asymptotic profiles of positive solutions in some limiting cases of large/small diffusion rates $\epsilon$ and $\mu$, see Theorem \ref{thm5} (large $\epsilon$), Theorem \ref{thm4} (large $\mu$) and Theorem \ref{thm6} (small $\mu$).  The uniqueness of solutions  is not only a difficult mathematical question for elliptic problems in general, but also predict the possible long time dynamics of the system, while the solution profiles describe the spatial distribution patterns of species. These results can be carried over to the original problem \eqref{2022032801} directly via the transform \eqref{trans}.

The rest of this paper is organized. In section 2, we shall study the eigenvalue problem of \eqref{KO2} and find the conditions for the stability/instability of the unique semi-trivial solution. Furthermore we establish some preliminary results for later use. In section 3, we employ the topological degree method (index theory) to show that the instability of the semi-trivial solutions ensures the existence of positive solutions and hence to   establish our main result on the existence/non-existence of positive solutions of \eqref{KO2}. In section 4, we prove the uniqueness and characterize the asymptotical profile of positive solutions of \eqref{KO2} for large/small diffusion rates $\epsilon$ and $\mu$.

\section{Stability of semi-trivial solutions}
In this section, we study the eigenvalue problem associated with the problem \eqref{KO2} and give some conditions for the stability/instability of the semi-trivial solution of \eqref{KO2}. We begin with the following linear eigenvalue problem
\begin{equation}\label{2022032803}
\begin{cases}
\ell\Delta \phi+r(x)\phi=\lambda\phi, &\mbox{in}\;\Omega,\\
\nabla \phi \cdot \n=0,&\mbox{on}\;\partial\Omega,
\end{cases}
\end{equation}
where $r\in C({\Omega})$. We denote the principal eigenvalue and eigenfunction by $\lambda_1(\ell,r)$ and $\phi_1(\ell,r)$, respectively, where one can choose $\phi_1(\ell,r)>0$ and $\|\phi_1(\ell,r)\|_{L^\infty}=1$ and there is no other eigenvalue with a positive eigenfunction \cite{KreinRutman1948}. Moreover, by the variational approach,
$\lambda_1(\ell,r)$ can be characterized as
\begin{equation}\label{2022061703}
\lambda_1(\ell,r)=\sup\limits_{0\neq\phi\in H^1(\Omega)}\frac{\int_\Omega(-\ell|\nabla \phi|^2+r\phi^2)\mathrm{d}x}{\int_\Omega\phi^2\mathrm{d} x}.
\end{equation}

Under the assumption $(H_1)$, it is straightforward to see (cf. \cite{CantrellCosner}) that system \eqref{KO2} admits a unique semi-trivial solution $(\tilde{u},0)$ for any $\epsilon>0$ , where $\tilde{u}>0$ satisfies
\begin{equation}\label{2022062101}
\begin{cases}
\epsilon\Delta \tilde{u}+\tilde{u}(m(x)-\tilde{u})=0, &\mbox{in}\;\Omega,\\
\nabla \tilde{u}\cdot \n=0,&\mbox{on}\;\partial\Omega.
\end{cases}
\end{equation}
The problem \eqref{2022062101} has been well studied in the literature and there are wealthy results available (cf. \cite{CantrellCosner, Ni-book-SIAM-2011}). Below we cite a result that shall be used later.

\begin{proposition}(\cite[Propositon 2.5]{LamNi})\label{prop2}
The problem \eqref{2022062101} has a unique positive solution $\tilde{u}$ satisfying
\begin{itemize}
\item[(i)] $\tilde{u}\to m^+=\max\{m,0\}$ in $L^\infty(\Omega)$ as $\epsilon\to0$;\vspace{0.1cm}

\item[(ii)] $\tilde{u}\to  \frac{1}{|\Omega|}\int_\Omega m\mathrm{d}x$ in $L^\infty(\Omega)$ as $\epsilon\to+\infty$.
\end{itemize}
\end{proposition}
We also collect some results on the principal eigenvalue and eigenfunction of \eqref{2022032803}.
\begin{lemma}\label{lem01}
If $r\in C(\Omega)$, then the following statements on the principal eigenvalue $\lambda_1(\ell,r)$ and eigenfunction $\phi_1(\ell,r)$ of problem~\eqref{2022032803} are true.
\begin{itemize}
\item[(i)] $\lambda_1(\ell,r)$ and $\phi_1(\ell,r)$ depend smoothly on parameters $\ell\in(0,+\infty)$ and continuously on $r\in C(\Omega)$.

\item[(ii)] If $r$ is constant on $(0,L)$, then $\lambda_1(\ell,r)=r$; otherwise, the principal eigenvalue $\lambda_1(\ell,r)$ is strictly decreasing with respect to $\ell\in(0,+\infty)$ and
\begin{equation}\label{2022082901}
\lim\limits_{\ell\to0}\lambda_1(\ell,r)=\max\limits_{x\in\bar{\Omega}}r(x)
\quad \mbox{and\quad $\lim\limits_{\ell\to+\infty}\lambda_1(\ell,r)=\frac{1}{|\Omega|}\int_\Omega r(x)\mathrm{d}x$.}
\end{equation}

\item[(iii)] If $r_i\in C(\Omega)$  ($i=1$, $2$)  and $r_1\geq,\not\equiv r_2$ in $\Omega$, then $\lambda_1(\ell,r_1)>\lambda_1(\ell,r_2)$.
\end{itemize}
\end{lemma}
\begin{proof}
The proofs of statements (i)-(iii) are quite standard. See, for example, \cite[Page 95 and Page 162]{CantrellCosner} and~\cite{LamNi}.
\end{proof}

Then, we define the notion of linear stability of a given steady state $(u,w)$.  The eigenvalue problem of the linearized system \eqref{KO2} at $(u,w)$, reads
\begin{equation}\label{2022061702}
\begin{cases}
\epsilon\Delta \phi+(m-2u)\phi-\left(\frac{F(u)}{d(u)}\right)' w\phi-\frac{F(u)}{d(u)}\psi=\tau\phi, &\mbox{in}\;\Omega,\\
\mu\Delta \psi+\alpha\left(\frac{F(u)}{d(u)}\right)'w\phi+\frac{\theta d'(u)}{d^2(u)}w\phi+\frac{\alpha F(u)-\theta}{d(u)}\psi  =\tau\psi, &\mbox{in}\;\Omega,\\[2mm]
\nabla \phi \cdot \n=\nabla \psi \cdot \n=0,&\mbox{on}\;\partial\Omega,
\end{cases}
\end{equation}
where $'$ denotes the differentiation with respect to $u$, and $(\phi,\psi)$ is the eigenfunction associated with the eigenvalue $\tau$.\\

Throughout the paper, the following convention will be adopted.

\begin{definition}
An eigenvalue $\tau_1$ of problem \eqref{2022061702} is called a principal eigenvalue if $\tau_1\in \mathbb{C}$ and for any eigenvalue $\tau$ with $\tau\neq\tau_1$, we have Re $\tau\leq$ Re $\tau_1$.
If Re $\tau_1<0$, then $(u,w)$ is linearly stable; while if Re $\tau_1>0$, then  $(u,w)$ is linearly unstable; we call $(u,w)$ is neutrally stable if Re $\tau_1=0$.
\end{definition}

We remark here that the principal eigenvalue of problem \eqref{2022061702} may not be unique but the real part of $\tau_1$ are equal. Following the approach as that in \cite[Lemma 2.9 and Corollary 2.10]{LamNi}, we can readily derive the following result and omit the details for brevity.
\begin{lemma}\label{lem02}
For system \eqref{KO2}, the following results hold.
\begin{itemize}
\item[(1)]
$(0,0)$ is linearly stable if and only if
$
\max\big\{\lambda_1(\epsilon,m),\lambda_1\big(\mu,-\frac{\theta}{d(0)}\big)\big\}<0.
$
\item[(2)] $(\tilde{u},0)$ is linearly stable if and only if
$
\max\big\{\lambda_1\big(\mu,\frac{\alpha F(\tilde{u})-\theta}{d(\tilde{u})}\big),\lambda_1(\epsilon,m-2\tilde{u})\big\}<0.
$
\end{itemize}
\end{lemma}

Based on Lemma \ref{lem01} and Lemma \ref{lem02}, we have the following result.
\begin{lemma}\label{lem03}
The trivial solution $(0,0)$ of \eqref{KO2} is linearly unstable.
\end{lemma}
\begin{proof}
From Lemma \ref{lem01} and assumption $(H_1)$, it follows that
$$
\lambda_1(\epsilon,m)>\lim\limits_{\mu\to+\infty}\lambda_1(\mu,m)=\frac{\int_\Omega m\mathrm{d}x}{|\Omega|}>0,
$$
which alongside Lemma \ref{lem02} and the fact $\lambda_1\big(\mu,-\frac{\theta}{d(0)}\big)=-\frac{\theta}{d(0)}<0$ shows that $(0,0)$ is linearly unstable.
\end{proof}

Next, we study the stability of semi-trivial solution $(\tilde{u},0)$ to system \eqref{KO2}.
\begin{lemma}\label{lem04*}
$(\tilde{u},0)$ is linearly stable if and only if
$\lambda_1\left(\mu,\frac{\alpha F(\tilde{u})-\theta}{d(\tilde{u})}\right)<0.$
\end{lemma}
\begin{proof}
From $\lambda_1(\epsilon,m-\tilde{u})=0$ and Lemma \ref{lem01} (iii), it follows that $\lambda_1(\epsilon,m-2\tilde{u})<0$. This alongside Lemma \ref{lem02} implies that $(\tilde{u},0)$ is linearly stable if and only if
$\lambda_1\left(\mu,\frac{\alpha F(\tilde{u})-\theta}{d(\tilde{u})}\right)<0.$
\end{proof}

In the sequel, in some cases, instead of general density-dependent dispersal rate function $d(u)$, we shall consider the following specialized forms for the definiteness
\begin{equation}\label{du}
d(u):=d(u;k)=e^{-ku} \ \text{or}\ (1+u)^{-k}, \ \ k\geq 0.
\end{equation}
Subsequent to this,  we shall denote
\begin{equation}\label{thetak}
\theta_0=\frac{\alpha}{|\Omega|} \int_\Omega F(\tilde{u})dx, \ \theta_k=\frac{\alpha \int_\Omega\frac{ F(\tilde{u})}{d(\tilde{u};k)}\mathrm{d}x}{\int_\Omega\frac{1}{d(\tilde{u};k)}\mathrm{d}x}\  \text{for} \ k> 0.
\end{equation}
We also denote
$$\tilde{u}_{\min}=\min\limits_{x\in\bar{\Omega}}\tilde{u}\ \ \mathrm{and}\ \  \tilde{u}_{\max}=\max\limits_{x\in\bar{\Omega}}\tilde{u}.$$ Then it follows from Lemma \ref{lem01} and assumption $(H_2)$ that $\alpha F(\tilde{u}_{\min})<\lambda_1(\mu,\alpha F(\tilde{u}))<\alpha F(\tilde{u}_{\max})$.
\\

Then we have the following key results.

\begin{lemma}\label{lem04}
There exists some $\tilde{\theta}\in(\alpha F(\tilde{u}_{\min}),\alpha F(\tilde{u}_{\max}))$ satisfying $\lambda_1\left(\mu,\frac{\alpha F(\tilde{u})-\tilde{\theta}}{d(\tilde{u})}\right)=0$ such that
\begin{equation}\label{2022062102}
(\tilde{u},0) \ \text{is}\ \begin{cases}
\mbox{linearly stable}& \ \text{if} \ \ \theta>\tilde{\theta},\\
\mbox{linearly unstable}& \ \text{if}\ \ 0\leq\theta<\tilde{\theta},
\end{cases}
\end{equation}
and hence
\begin{equation}\label{2022062103}
(\tilde{u},0)\ \text{is}\ \begin{cases}
\mbox{linearly stable}& \ \text{if} \ \ \theta\geq\alpha F(\tilde{u}_{\max}),\\
\mbox{linearly unstable}& \ \text{if} \ \ 0\leq\theta\leq\alpha F(\tilde{u}_{\min}).
\end{cases}
\end{equation}
Moreover, for any $\theta\in(\alpha F(\tilde{u}_{\min}),\alpha F(\tilde{u}_{\max}))$, if $d(u)=d(u;k)$, where $d(u;k)=e^{-ku}$ or $(1+u)^{-k}$ with $k\geq 0$,  the following results on the linear stability/instability of $(\tilde{u},0)$ hold true.
\begin{itemize}
\item[(i)]
{
Fixing all the parameters except $\mu$, if $\theta\in\left(\alpha F(\tilde{u}_{\min}),\theta_k\right]$, then $(\tilde{u},0)$ is linearly unstable for any $\mu>0$; while if
$\theta\in\left(\theta_k,\alpha F(\tilde{u}_{\max})\right)$, then there exits some $\mu^*>0$ (depending on $k$ and $\theta$) satisfying $\lambda_1\left(\mu^*,\frac{\alpha F(\tilde{u})-\theta}{d(\tilde{u};k)}\right)=0$ such that
$$
(\tilde{u},0)\ \text{is}\ \begin{cases}
\mbox{linearly stable}& \ \text{if} \ \ {\mu} >\mu^*,\\
\mbox{linearly unstable}& \ \text{if} \ \ 0<{\mu} <\mu^*.
\end{cases}
$$
}

\item[(ii)]
{
Fix all the parameters except ${\mu} $ and $k$.  We have the following statements.
\begin{itemize}
\item[(ii.1)]
{If $\theta\in\left(\alpha F(\tilde{u}_{\min}),\theta_0\right]$, then $(\tilde{u},0)$ is linearly unstable for any ${\mu} >0$ and $k\geq0$.
    }

\item[(ii.2)]
{If $\theta\in\left[\theta_0,\alpha F(\tilde{u}_{\max})\right)$, there exists $k^*(\theta)>0$ satisfying
\begin{equation}\label{2022062802}
\mathrm{sgn}\bigg(\int_\Omega\frac{\alpha F(\tilde{u})-\theta}{d(\tilde{u};k)}\mathrm{d}x\bigg)=\mathrm{sgn}(k-k^*)
\end{equation}
such that $(\tilde{u},0)$ is linearly unstable for any ${\mu} >0$ provided that $k\geq k^*$. Moreover, there exists some $\tilde{\mu}$ such that $(\tilde{u},0)$ is linearly unstable for any $k\in[0,k^*)$ and ${\mu} \in(0,\tilde{\mu})$.}
\end{itemize}
}
\end{itemize}
\end{lemma}
\begin{proof}
From Lemma \ref{lem01} (iii) and $d(\tilde{u})>0$ in $\Omega$, it follows that $\lambda_1\left(\mu,\frac{\alpha F(\tilde{u})-\theta}{d(\tilde{u})}\right)$ is strictly decreasing with respect to $\theta$. Therefore, it suffices to consider the values of $\lambda_1\Big(\mu,\frac{\alpha F(\tilde{u})}{d(\tilde{u})}\Big)$ and $\lim\limits_{\theta\to+\infty}\lambda_1\left(\mu,\frac{\alpha F(\tilde{u})-\theta}{d(\tilde{u})}\right)$. Based on the variational formula \eqref{2022061703}, one has
$$
\lim\limits_{\theta\to+\infty}\lambda_1\left(\mu,\frac{\alpha F(\tilde{u})-\theta}{d(\tilde{u})}\right)=-\infty,
$$
and
\begin{equation}\nonumber
\aligned
\lambda_1\left(\mu,\frac{\alpha F(\tilde{u})}{d(\tilde{u})}\right)&=\sup\limits_{0\neq\phi\in H^1(\Omega)}\frac{\int_\Omega\left(-\mu|\nabla \phi|^2+\frac{\alpha F(\tilde{u})}{d(\tilde{u})}\phi^2\right)\mathrm{d}x}{\int_\Omega\phi^2\mathrm{d} x}
\geq \frac{\int_{\Omega}\frac{\alpha F(\tilde{u})}{d(\tilde{u})}\mathrm{d}x}{|\Omega|}
>0,
\endaligned
\end{equation}
which suggests that there exists some $\tilde{\theta}\in(0,+\infty)$ such that $\lambda_1\left(\mu,\frac{\alpha F(\tilde{u})-\tilde{\theta}}{d(\tilde{u})}\right)=0$. To prove that $\tilde{\theta}\in(\alpha F(\tilde{u}_{\min}),\alpha F(\tilde{u}_{\max}))$, it suffices to show that
\begin{equation}\label{2022062104}
\lambda_1\left(\mu,\frac{\alpha F(\tilde{u})-\alpha F(\tilde{u}_{\min})}{d(\tilde{u})}\right)>0\quad\mbox{and}\quad\lambda_1\left(\mu,\frac{\alpha F(\tilde{u})-\alpha F(\tilde{u}_{\max})}{d(\tilde{u})}\right)<0.
\end{equation}
Recalling the assumption $(H_2)$ and the fact $\tilde{u}$ is not a constant function in $\Omega$, one obtains
$$
\alpha F(\tilde{u})-\alpha F(\tilde{u}_{\min})\geq,\not\equiv0\quad\mbox{and}\quad\alpha F(\tilde{u})-\alpha F(\tilde{u}_{\max})\leq,\not\equiv0,
$$
which combined with Lemma \ref{lem01} (iii) implies that \eqref{2022062104} holds. Therefore, \eqref{2022062102} and \eqref{2022062103} hold.

Next, we consider the case $\theta\in(\alpha F(\tilde{u}_{\min}),\alpha F(\tilde{u}_{\max}))$.
Since the proofs are similar, we only consider the case $d(u)=e^{-ku}$.
By  Lemma \ref{lem01} (ii), one obtains
\begin{eqnarray}\label{0}
\lim\limits_{{\mu} \to0}\lambda_1({\mu} ,(\alpha F(\tilde{u})-\theta)e^{k\tilde{u}})=\max\limits_{x\in\bar{\Omega}}(\alpha F(\tilde{u})-\theta)e^{k\tilde{u}})>0.
\end{eqnarray}
To proceed, we recall the notation  $\theta_k=\frac{\int_\Omega\alpha F(\tilde{u})e^{k\tilde{u}}\mathrm{d}x}{\int_\Omega e^{k\tilde{u}}\mathrm{d}x}$. Then clearly $\alpha F(\tilde{u}_{\min})<\theta_k < \alpha F(\tilde{u}_{\max})$ for any $k\geq 0$. Using \eqref{2022082901}, one has
\begin{eqnarray}\label{infty}
\lim\limits_{{\mu} \to+\infty}\lambda_1({\mu} ,(\alpha F(\tilde{u})-\theta)e^{k\tilde{u}})=
\frac{\int_\Omega(\alpha F(\tilde{u})-\theta)e^{k\tilde{u}}\mathrm{d}x}{|\Omega|}     \begin{cases}
     \geq 0,&\mbox{if}\;\theta\leq\theta_k,\\
     <0,&\mbox{if}\;\theta>\theta_k.
     \end{cases}
\end{eqnarray}
From Lemma \ref{lem01} (i)-(ii), it follows that the principal eigenvalue $\lambda_1\left({\mu} ,(\alpha F(\tilde{u})-\theta)e^{k\tilde{u}}\right)$ smoothly depends on ${\mu} $ and is strictly decreasing with respect to ${\mu} \in(0,+\infty)$.
Therefore from \eqref{0}-\eqref{infty}, when $\theta\leq\theta_k$, we have $\lambda_1\left(\mu,(\alpha F(\tilde{u})-\theta)e^{k\tilde{u}}\right)>0$ for any ${\mu} >0$, which indicates that $(\tilde{u},0)$ is linearly unstable from Lemma \ref{lem04*}. As $\theta>\theta_k$, from \eqref{0}-\eqref{infty}, we find a constant $\mu^*>0$ such that $\lambda_1\left(\mu^*,\frac{\alpha F(\tilde{u})-\theta}{d(\tilde{u};k)}\right)=0$ and $\lambda_1\left(\mu,(\alpha F(\tilde{u})-\theta)e^{k\tilde{u}}\right)>0$ if $0<\mu<\mu^*$ while $\lambda_1\left(\mu,(\alpha F(\tilde{u})-\theta)e^{k\tilde{u}}\right)<0$ if $\mu>\mu^*$. This completes the proof of  statement (i) by using Lemma \ref{lem04*}.

Next we  prove the results in statement (ii). To this end,  we first prove a claim as below.

{\bf Claim 1}:  $\int_\Omega(\alpha F(\tilde{u})-\rho)e^{k\tilde{u}}\mathrm{d}x$ is strictly increasing with respect to $k\in[0,+\infty)$ provided $\rho\leq\theta_0=: \frac{\int_\Omega\alpha F(\tilde{u})dx}{|\Omega|}$. To prove this,
for any $\rho\leq\rho_0$, we define $\mathfrak{f}(k)=\int_\Omega(\alpha F(\tilde{u})-\rho)e^{k\tilde{u}}\mathrm{d}x$. Direct computations show that
\begin{equation}\label{2022062801}
\aligned
\mathfrak{f}'(k)&=\int_\Omega\tilde{u}(\alpha F(\tilde{u})-\rho)e^{k\tilde{u}}\mathrm{d}x\\
&=\int_{\{x\in\Omega|\tilde{u}(x)\geq F^{-1}\left(\frac{\rho}{\alpha}\right)\}}\tilde{u}(\alpha F(\tilde{u})-\rho)e^{k\tilde{u}}\mathrm{d}x+\int_{\{x\in\Omega|\tilde{u}(x)<F^{-1}\left(\frac{\rho}{\alpha}\right)\}}\tilde{u}(\alpha F(\tilde{u})-\rho)e^{k\tilde{u}}\mathrm{d}x\\
&>\int_{\{x\in\Omega|\tilde{u}(x)\geq F^{-1}\left(\frac{\rho}{\alpha}\right)\}}F^{-1}\left(\frac{\rho}{\alpha}\right)(\alpha F(\tilde{u})-\rho)e^{kF^{-1}\left(\frac{\rho}{\alpha}\right)}\mathrm{d}x\\
&\quad\quad+\int_{\{x\in\Omega|\tilde{u}(x)<F^{-1}\left(\frac{\rho}{\alpha}\right)\}}F^{-1}\left(\frac{\rho}{\alpha}\right)(\alpha F(\tilde{u})-\rho)e^{kF^{-1}\left(\frac{\rho}{\alpha}\right)}\mathrm{d}x\\
&=\int_{\Omega}F^{-1}\left(\frac{\rho}{\alpha}\right)(\alpha F(\tilde{u})-\rho)e^{kF^{-1}\left(\frac{\rho}{\alpha}\right)}\mathrm{d}x\geq0,
\endaligned
\end{equation}
where we have used  the assumption $(H_2)$. Therefore,  Claim 1 holds.

If $\theta\in\left(\alpha F(\tilde{u}_{\min}),\theta_0\right]$, using claim 1, then we have
$$
\aligned
\int_\Omega(\alpha F(\tilde{u})-\theta)e^{k\tilde{u}}\mathrm{d}x&\geq \int_\Omega\left(\alpha F(\tilde{u})-\theta_0\right)e^{k\tilde{u}}\mathrm{d}x
\geq\int_\Omega\left(\alpha F(\tilde{u})-\theta_0\right)\mathrm{d}x=0,
\endaligned
$$
which together with  Lemma \ref{lem01} (ii) and Lemma \ref{lem04*} implies that $(\tilde{u},0)$ is linearly unstable for any ${\mu} >0$ and $k\geq0$. This proves the results in (ii.1).

If $\theta\in\left(\theta_0,\alpha F(\tilde{u}_{\max})\right)$, we define $\mathcal{F}(k):=\int_\Omega(\alpha F(\tilde{u})-\theta)e^{k\tilde{u}}\mathrm{d}x$ as a function of the parameter $k$. Similar to the arguments as those in \cite[Lemma 2.6]{TangWang}, one can derive that
\begin{equation}\label{2022062602}
\lim\limits_{k\to+\infty}\frac{\int_{\Omega}e^{k\tilde{u}}F(\tilde{u})\mathrm{d}x}{\int_\Omega e^{k\tilde{u}}\mathrm{d}x}=\lim\limits_{k\to+\infty}\frac{\int_{\Omega}e^{k(\tilde{u}-\tilde{u}_{\max})}F(\tilde{u})\mathrm{d}x}{\int_\Omega e^{k(\tilde{u}-\tilde{u}_{\max})}\mathrm{d}x}=F(\tilde{u}_{\max}),
\end{equation}
which along with the assumption $\theta\in\left(\theta_0,\alpha F(\tilde{u}_{\max})\right)$ gives that
$$
\mathcal{F}(0)<0\quad\mbox{and}\quad\mathcal{F}(\infty)>0.
$$
This together with the continuity of $\mathcal{F}(\cdot)$ yields a constant  $k^*>0$ such that $\mathcal{F}(k^*)=\int_\Omega(\alpha F(\tilde{u})-\theta)e^{k^*\tilde{u}}\mathrm{d}x=0$. Next we prove that the positive root of
$\mathcal{F}(k)=0$ is unique, for which it suffices to show  the following claim.

{\bf Claim 2}: For any $k_0>0$ satisfying $\mathcal{F}(k_0)=0$, then $\mathcal{F}'(k_0)>0$. Indeed similar to \eqref{2022062801}, one can deduce that
$$
\aligned
\mathcal{F}'({k_0})&=\int_\Omega\tilde{u}(\alpha F(\tilde{u})-\theta)e^{{k_0}\tilde{u}}\mathrm{d}x\\
&=\int_{\{x\in\Omega|\tilde{u}(x)\geq F^{-1}(\theta/\alpha)\}}\tilde{u}(\alpha F(\tilde{u})-\theta)e^{{k_0}\tilde{u}}\mathrm{d}x+\int_{\{x\in\Omega|\tilde{u}(x)<F^{-1}(\theta/\alpha)\}}\tilde{u}(\alpha F(\tilde{u})-\theta)e^{{k_0}\tilde{u}}\mathrm{d}x\\
&>\int_{\{x\in\Omega|\tilde{u}(x)\geq F^{-1}(\theta/\alpha)\}}F^{-1}(\theta/\alpha)(\alpha F(\tilde{u})-\theta)e^{{k_0}\tilde{u}}\mathrm{d}x\\
&\quad\quad+\int_{\{x\in\Omega|\tilde{u}(x)<F^{-1}(\theta/\alpha)\}}F^{-1}(\theta/\alpha)(\alpha F(\tilde{u})-\theta)e^{{k_0}\tilde{u}}\mathrm{d}x\\
&=\int_{\Omega}F^{-1}(\theta/\alpha)(\alpha F(\tilde{u})-\theta)e^{{k_0}\tilde{u}}\mathrm{d}x\\
&=F^{-1}(\theta/\alpha)\mathcal{F}(k_0)=0.
\endaligned
$$
Therefore, $k^*$ is the unique positive root of $\mathcal{F}(k)=0$ and hence \eqref{2022062802} holds. Combining the facts in \eqref{2022062802}, Lemma \ref{lem01} and Lemma \ref{lem04*}, one concludes that $(\tilde{u},0)$ is linearly unstable for any ${\mu} >0$ and $k\geq k^*$. This shows the first part (ii.1) of assertion(ii). From the results in statement (i), one has that
$$
\tilde{\mu}=\inf\limits_{k\in[0,k^*)}\mu^*(k)>0,
$$
which directly implies the second part (ii.2) of statement (ii) by the same argument as in the proof of statement (i). This complete the proof.
\end{proof}

\begin{remark}
Fixing all the parameters except $\mu$ and $\theta$, for any $\theta\in\left(\theta_k,\alpha F(\tilde{u}_{\max})\right)$, Lemma \ref{lem04} (i) ensures a number $\mu^*(\theta)>0$ such that $\lambda_1\left(\mu^*(\theta),\frac{\alpha F(\tilde{u})-\theta}{d(\tilde{u};k)}\right)=0$. We will show that $\mu^*(\theta)$ is a convex function with respect to $\theta\in\left(\theta_k,\alpha F(\tilde{u}_{\max})\right)$ by proving that
\begin{equation}\label{2022091101}
\rho\mu^*(\theta_1)+(1-\rho)\mu^*(\theta_2)>\mu^*(\rho\theta_1+(1-\rho)\theta_2),
\end{equation}
for any $\rho\in(0,1)$ and $\theta_i\in\left(\theta_k,\alpha F(\tilde{u}_{\max})\right)$ $(i=1,2)$. From \eqref{2022061703}, $\lambda_1\left(\mu^*(\theta_1),\frac{\alpha F(\tilde{u})-\theta_1}{d(\tilde{u};k)}\right)=\lambda_1\left(\mu^*(\theta_2),\frac{\alpha F(\tilde{u})-\theta_2}{d(\tilde{u};k)}\right)=0$, and $\phi_1\left(\mu^*(\theta_1),\frac{\alpha F(\tilde{u})-\theta_1}{d(\tilde{u};k)}\right)\neq\phi_1\left(\mu^*(\theta_2),\frac{\alpha F(\tilde{u})-\theta_2}{d(\tilde{u};k)}\right)$, it follows that
$$
\aligned
0&=\lambda_1\left(\mu^*(\rho\theta_1+(1-\rho)\theta_2),\frac{\alpha F(\tilde{u})-(\rho\theta_1+(1-\rho)\theta_2)}{d(\tilde{u};k)}\right)\\
&=\sup\limits_{0\neq\phi\in H^1(\Omega)}\frac{\int_\Omega\left(-\mu^*(\rho\theta_1+(1-\rho)\theta_2)|\nabla\phi|^2+\frac{\alpha F(\tilde{u})-(\rho\theta_1+(1-\rho)\theta_2)}{d(\tilde{u};k)}\phi^2\right)\mathrm{d}x}{\int_\Omega\phi^2\mathrm{d}x}\\
&< \sup\limits_{0\neq\phi\in H^1(\Omega)} \frac{\int_\Omega\rho\left(-\mu^*(\theta_1)|\nabla\phi|^2+\frac{\alpha F(\tilde{u})-\theta_1}{d(\tilde{u};k)}\right)\mathrm{d}x}{\int_\Omega\phi^2\mathrm{d}x}\\
&\ \ \ +\sup\limits_{0\neq\phi\in H^1(\Omega)} \frac{\int_\Omega(1-\rho)\left(-\mu^*(\theta_2)|\nabla\phi|^2+\frac{\alpha F(\tilde{u})-\theta_2}{d(\tilde{u};k)}\right)\mathrm{d}x}{\int_\Omega\phi^2\mathrm{d}x}\\
&\ \ \ +\sup\limits_{0\neq\phi\in H^1(\Omega)}\frac{\int_\Omega(\rho\mu^*(\theta_1)+(1-\rho)\mu^*(\theta_2)-\mu^*(\rho\theta_1+(1-\rho)\theta_2))|\nabla \phi|^2\mathrm{d}x}{\int_\Omega\phi^2\mathrm{d}x}\\
&= (\rho\mu^*(\theta_1)+(1-\rho)\mu^*(\theta_2)-\mu^*(\rho\theta_1+(1-\rho)\theta_2))\sup\limits_{0\neq\phi\in H^1(\Omega)}\frac{\int_\Omega|\nabla \phi|^2\mathrm{d}x}{\int_\Omega\phi^2\mathrm{d}x},
\endaligned
$$
which implies that \eqref{2022091101} holds.
\end{remark}
To proceed, we present a generalized result of \cite[Lemma 26]{DeAngelis} below, which can be proved readily by the mathematical induction.
\begin{proposition}\label{prop01*}
Suppose there are three sequences of nonnegative real numbers such that
$0\leq a_1\leq a_2\leq \cdots\leq a_n$, $0\leq b_1\leq b_2\leq \cdots\leq b_n$ and
$0\leq c_1\leq c_2\leq \cdots\leq c_n$. Then
\begin{equation}\label{2022091102}
\left(\sum_{j=1}^{n}a_jb_jc_j\right)\left(\sum_{j=1}^{n}c_j\right)\geq \left(\sum_{j=1}^{n}a_jc_j\right)\left(\sum_{j=1}^{n}b_jc_j\right)
\end{equation}
where ``='' holds if and only if $a_1=a_n$ or $b_1=b_n$ or $c_1=c_{n-1}=0$.
\end{proposition}
\begin{proof}
We use induction. If $n = 1$, then \eqref{2022091102} holds. Now we assume that \eqref{2022091102} holds when $n=i$, we need to show that it holds for $n=i+1$. Direct computations give
$$
\aligned
&(a_1b_1c_1+a_2b_2c_2+\cdots+a_ib_ic_i+a_{i+1}b_{i+1}c_{i+1})(c_1+c_2+\cdots+c_i+c_{i+1})\\
&\textstyle \quad=\left(\sum_{j=1}^{i}a_jb_jc_j\right)\left(\sum_{j=1}^{i}c_j\right)+a_{i+1}b_{i+1}c_{i+1}^2+c_{i+1}\left(\sum_{j=1}^{i}a_jb_jc_j\right)+a_{i+1}b_{i+1}c_{i+1}\left(\sum_{j=1}^{i}c_j\right)\\
&\textstyle \quad\geq\left(\sum_{j=1}^{i}a_jc_j\right)\left(\sum_{j=1}^{i}b_jc_j\right)+a_{i+1}b_{i+1}c_{i+1}^2+\left(\sum_{j=1}^{i}c_jc_{i+1}(a_jb_j+a_{i+1}b_{i+1})\right)\\
&\textstyle\quad=\left(\sum_{j=1}^{i}a_jc_j\right)\left(\sum_{j=1}^{i}b_jc_j\right)+a_{i+1}b_{i+1}c_{i+1}^2+\left(\sum_{j=1}^{i}c_jc_{i+1}(a_jb_{i+1}+b_ja_{i+1})\right)\\
&\textstyle\quad\quad+\left(\sum_{j=1}^{i}c_jc_{i+1}(a_jb_j+a_{i+1}b_{i+1}-a_jb_{i+1}-b_ja_{i+1})\right)\\
&\textstyle\quad=\left(\sum_{j=1}^{i}a_jc_j\right)\left(\sum_{j=1}^{i}b_jc_j\right)+a_{i+1}b_{i+1}c_{i+1}^2+b_{i+1}c_{i+1}\left(\sum_{j=1}^{i}a_jc_j\right)+a_{i+1}c_{i+1}\left(\sum_{j=1}^{i}b_jc_j\right)\\
&\textstyle\quad\quad+\left(\sum_{j=1}^{i}c_jc_{i+1}(b_j-b_{i+1})(a_j-a_{i+1})\right)\\
&\textstyle\quad\geq\left(\sum_{j=1}^{i+1}a_jc_j\right)\left(\sum_{j=1}^{i+1}b_jc_j\right).
\endaligned
$$
where the ``='' in the last inequality holds if and only if $a_j=a_{i+1}$ or $b_j=b_{i+1}$ or $c_j=0$ for all $j=1,2,\cdots, i$. This along with the fact that $a_j, b_i, c_j$ are non-decreasing with respect to $j$ completes the proof.
\end{proof}
We remark that the results in Proposition \ref{prop01*} can be considered as a generalization of \cite[Lemma 26]{DeAngelis} where $c_i=1$ $(i=1,2,\cdots,n)$.

\begin{lemma}\label{lem01*}
Let $d(u)=:d(u;k)$ be given in \eqref{du}. Fix all the parameters except $k$ and define $
\tilde{F}(k):=\frac{\int_\Omega\frac{ F(\tilde{u})}{d(\tilde{u};k)}\mathrm{d}x}{\int_\Omega\frac{1}{d(\tilde{u};k)}\mathrm{d}x}$. Then $\tilde{F}(k)$ is strictly increasing with respect to $k\in[0,+\infty)$.
\end{lemma}
\begin{proof}
We first consider the case $d(u;k)=e^{-ku}$ for which one has
$$
\tilde{F}'(k)=\frac{\int_\Omega e^{k\tilde{u}}\mathrm{d}x\int_\Omega F(\tilde{u})\tilde{u}e^{k\tilde{u}}\mathrm{d}x-\int_\Omega\tilde{u}e^{k\tilde{u}}\mathrm{d}x\int_\Omega F(\tilde{u})e^{k\tilde{u}}\mathrm{d}x}{(\int_\Omega e^{k\tilde{u}}\mathrm{d}x)^2}.
$$
Next we shall approximate the integrals by their Riemann sums with
$$
a_i=F(\tilde{u}(x_i)),\;\;b_i=\tilde{u}(x_i),\;\;\mbox{and}\;\;c_i=e^{k\tilde{u}(x_i)},\quad i=1,2,\cdots,n.
$$
Since we can rearrange the terms in the Riemann sums in the order that $b_i$ is ascending
(then $a_i$ and $c_i$ are automatically ascending by the assumption $(H_2)$), by \eqref{2022091102}, one obtains
$$
\left(\frac{1}{n}\sum_{i=1}^{n}F(\tilde{u}(x_i))\tilde{u}(x_i)e^{k\tilde{u}(x_i)}\right)\left(\frac{1}{n}\sum_{i=1}^{n}e^{k\tilde{u}(x_i)}\right)> \left(\frac{1}{n}\sum_{i=1}^{n}F(\tilde{u}(x_i))e^{k\tilde{u}(x_i)}\right)\left(\frac{1}{n}\sum_{i=1}^{n}\tilde{u}(x_i)e^{k\tilde{u}(x_i)}\right)
$$
where the strict inequality results from the fact that $\tilde{u}$ is not a constant function in $\Omega$ (cf. Proposition \ref{prop2}).
Thus, one has
$$
\tilde{F}'(k)>0,\quad\mbox{for any}\;k\geq0.
$$

On the other hand,  if $d(u;k)=(1+u)^{-k}$, then we have
$$
\tilde{F}'(k)=\frac{\int_\Omega (1+\tilde{u})^k\mathrm{d}x\int_\Omega F(\tilde{u})(1+\tilde{u})^k\ln(1+\tilde{u})\mathrm{d}x-\int_\Omega(1+\tilde{u})^k\ln(1+\tilde{u})\mathrm{d}x\int_\Omega F(\tilde{u})(1+\tilde{u})^k\mathrm{d}x}{(\int_\Omega (1+\tilde{u})^k\mathrm{d}x)^2}.
$$
Let
$$
a_i=F(\tilde{u}(x_i)),\;\;b_i=\ln(1+\tilde{u}(x_i)),\;\;\mbox{and}\;\;c_i=(1+\tilde{u}(x_i))^k,\quad i=1,2,\cdots,n.
$$
Similarly, one can show that $\tilde{F}'(k)>0$ for any $k\geq0$, which completes the proof.
\end{proof}

By Lemma \ref{lem04*}, the linear stability of the semi-trivial steady state $(\tilde{u},0)$ is determined by the sign of $\textstyle \lambda_1\left({\mu},\frac{\alpha F(\tilde{u})-\theta}{d(\tilde{u};k)} \right)$. Then, it is natural to study the level set
 $$\mathcal{S}_0:=\left\{({\mu} ,k)|\lambda_1\left({\mu} ,\frac{\alpha F(\tilde{u})-\theta}{d(\tilde{u};k)} \right)=0\right\}.
 $$
Fixing all the parameters except ${\mu} $ and $k$, if $\theta\in\left[\theta_0,\alpha F(\tilde{u}_{\max})\right)$, for any $k\in[0,k^*)$, from Lemma \ref{lem04} (ii), it follows that there exists unique $\mu^*(k)>0$ such that $\textstyle \lambda_1\left(\mu^*(k),\frac{\alpha F(\tilde{u})-\theta}{d(\tilde{u};k)} \right)=0$.  Next, we investigate the property of $\mu^*(k)$ by varying $k$ from $0$ to $k^*$, that is, to characterize the level set $\mathcal{S}_0$.
\begin{lemma}\label{lem05*}
Let $d(u)=d(u;k)$ be given by \eqref{du} and all the parameters except for ${\mu} $ and $k$ fixed. Let $k^*$ and $\tilde{\mu}$ be as defined in Lemma \ref{lem04}-(ii.2).  Assume $\theta\in\left(\theta_0,\alpha F(\tilde{u}_{\max})\right)$. For any $k_0\in[0,k^*)$, we have
\begin{equation}\label{2022062805}
\frac{\partial \lambda_1\left({\mu} ,\frac{\alpha F(\tilde{u})-\theta}{d(\tilde{u};k)}\right)}{\partial k}
=\frac{-\int_\Omega\frac{\alpha F(\tilde{u})-\theta}{d^2(\tilde{u};k)}\cdot\frac{\partial d(\tilde{u};k)}{\partial k}\phi_1^2\mathrm{d}x}{\int_\Omega\phi_1^2\mathrm{d}x},
\end{equation}
where $\phi_1=\phi_1\left({\mu} ,\frac{\alpha F(\tilde{u})-\theta}{d(\tilde{u};k)}\right)$. In particular,
\begin{equation}\label{2022062806}
\frac{\partial \lambda_1\left({\mu} ,(\alpha F(\tilde{u})-\theta)e^{k\tilde{u}}\right)}{\partial k}\Big|_{({\mu} ,k)=(\mu^*(k_0),k_0)}
=\mu^*(k_0)\frac{\int_\Omega\left[\tilde{u}|\nabla\phi|^2+\phi\nabla \phi\cdot\nabla \tilde{u}\right]\mathrm{d}x}{\int_{\Omega}\phi^2\mathrm{d}x},
\end{equation}
where $\phi=\phi_1(\mu^*(k_0),(\alpha F(\tilde{u})-\theta)e^{k_0\tilde{u}})$, and
\begin{equation}\label{2022062807}
\frac{\partial \lambda_1\left({\mu} ,(\alpha F(\tilde{u})-\theta)(1+\tilde{u})^k\right)}{\partial k}\Big|_{({\mu} ,k)=(\mu^*(k_0),k_0)}=\mu^*(k_0)\frac{\int_\Omega\left[\ln(1+\tilde{u})|\nabla\phi|^2+\frac{\phi\nabla \phi\cdot\nabla \tilde{u}}{1+\tilde{u}}\right]\mathrm{d}x}{\int_{\Omega}\phi^2\mathrm{d}x},
\end{equation}
where $\phi=\phi_1(\mu^*(k_0),(\alpha F(\tilde{u})-\theta)(1+\tilde{u})^{k_0})$.
If
\begin{equation}\label{2022062808}
\frac{\partial \lambda_1\left({\mu} ,\frac{\alpha F(\tilde{u})-\theta}{d(\tilde{u};k)}\right)}{\partial k}\Big|_{({\mu} ,k)=(\mu^*(k_0),k_0)}>0\ (\mathrm{resp}. <0),
\end{equation}
then $\frac{\partial \mu^*(k)}{\partial k}|_{k=k_0}>0\ (\mathrm{resp}. <0)$. Moreover, $\lim\limits_{k\to0}\mu^*(k)=\mu^*(0)$, $\lim\limits_{k\nearrow k^*}\mu^*(k)=+\infty$ and $\mu^*(k)\in(\tilde{\mu},+\infty)$ for any $k\in[0,k^*)$.
\end{lemma}
\begin{proof}
For simplicity, we denote $\lambda_1\left({\mu} ,\frac{\alpha F(\tilde{u})-\theta}{d(\tilde{u};k)}\right)$ and $\phi_1\left({\mu} ,\frac{\alpha F(\tilde{u})-\theta}{d(\tilde{u};k)}\right)$ by $\lambda_1$ and $\phi_1$, respectively. Recall that $\lambda_1$ and $\phi_1$ satisfy
\begin{equation}\label{2022062803}
\begin{cases}
{\mu} \Delta \phi_1+\frac{\alpha F(\tilde{u})-\theta}{d(\tilde{u};k)}\phi_1=\lambda_1\phi_1, &\mbox{in}\;\Omega,\\
\nabla \phi_1\cdot \n=0,&\mbox{on}\;\partial\Omega.
\end{cases}
\end{equation}
Differentiating (\ref{2022062803}) with respect to $k$, we get
\begin{equation}\label{2022062804}
\begin{cases}
 {\mu} \Delta \phi'_1+\frac{\alpha F(\tilde{u})-\theta}{d(\tilde{u};k)}\phi_1'-d'(\tilde{u};k)\frac{\alpha F(\tilde{u})-\theta}{d^2(\tilde{u};k)}\phi_1=\lambda_1 \phi'_1+\lambda_1'\phi_1, & x\in \Omega,\\
 \nabla \phi'_1\cdot \n=0, & x\in \partial\Omega
 \end{cases}
\end{equation}
where we have used $'$ to denote $\frac{\partial}{\partial k}$. Multiplying the first equation of (\ref{2022062803}) by $\phi_1'$, and then integrating the resulting equation on $\Omega$, one obtains
$$
\int_{\Omega}\left( {\mu} \phi'_1\Delta\phi_1 +\frac{\alpha F(\tilde{u})-\theta}{d(\tilde{u};k)}\phi_1\phi_1'\right)\textrm{d}x=\lambda_1\int_{\Omega}\phi_1\phi_1'\mathrm{d}x.
$$
Similarly, multiplying the first equation of (\ref{2022062804}) by $\phi_1$, and  integrating the resulting equation on $\Omega$, we obtain
$$
\int_{\Omega}\left({\mu} \phi_1\Delta\phi'_1 +\frac{(\alpha F(\tilde{u})-\theta)\phi_1\phi_1'}{d(\tilde{u};k)}-\frac{(\alpha F(\tilde{u})-\theta)\phi^2_1d'(\tilde{u};k)}{d^2(\tilde{u};k)}\right)\mathrm{d}x
=\lambda_1\int_{\Omega}\phi_1\phi'_1\mathrm{d}x+\lambda'_1\int_{\Omega}\phi_1^2\mathrm{d}x.
$$
Subtracting the above two equations and applying the integration by parts immediately give \eqref{2022062805}. Since the proofs of \eqref{2022062806} and  \eqref{2022062807} are similar, we only prove \eqref{2022062806}. Recall from Lemma \ref{lem04} that $(\lambda_1(\mu^*(k_0),(\alpha F(\tilde{u})-\theta)e^{k_0\tilde{u}}),\phi_1(\mu^*(k_0),(\alpha F(\tilde{u})-\theta)e^{k_0\tilde{u}}))=:(\lambda_1, \phi_1)=(0,\phi_1)$ satisfies
\begin{equation}\label{2022062803}
\begin{cases}
\mu^*(k_0) \Delta \phi_1+(\alpha F(\tilde{u})-\theta)e^{k_0 \tilde{u}}\phi_1=0, &\mbox{in}\;\Omega,\\
\nabla \phi_1\cdot \n=0,&\mbox{on}\;\partial\Omega.
\end{cases}
\end{equation}
Then it follows from \eqref{2022062805}  that
$$
\aligned
\frac{\partial \lambda_1\left({\mu} ,(\alpha F(\tilde{u})-\theta)e^{k\tilde{u}}\right)}{\partial k}\Big|_{({\mu} ,k)=(\mu^*(k_0),k_0)}
&=\frac{\int_\Omega(\alpha F(\tilde{u})-\theta)\tilde{u}e^{k_0\tilde{u}}\phi_1^2\mathrm{d}x}{\int_\Omega\phi_1^2\mathrm{d}x}\\
&=\frac{-\int_\Omega \mu^*(k_0)\phi_1\tilde{u}\Delta\phi_1\mathrm{d}x}{\int_\Omega\phi_1^2\mathrm{d}x}\\
&=\mu^*(k_0)\frac{\int_\Omega\left[\tilde{u}|\nabla\phi_1|^2+\phi_1\nabla \phi_1\cdot\nabla \tilde{u}\right]\mathrm{d}x}{\int_{\Omega}\phi_1^2\mathrm{d}x}.
\endaligned
$$
This proves \eqref{2022062806}.

As to \eqref{2022062808}, we only consider the case $\frac{\partial \lambda_1\left({\mu} ,\frac{\alpha F(\tilde{u})-\theta}{d(\tilde{u};k)}\right)}{\partial k}\Big|_{({\mu} ,k)=(\mu^*(k_0),k_0)}>0$ and the other case can been treated similarly. For this case, we recall that $\lambda_1(\mu^*(k_0),\frac{\alpha F(\tilde{u})-\theta}{d(k_0;\tilde{u})})=0$, which yields that (differentiate it with respect to $k$)
$$
\frac{\partial \lambda_1\left({\mu} ,\frac{\alpha F(\tilde{u})-\theta}{d(\tilde{u};k)}\right)}{\partial {\mu} }\bigg|_{({\mu} ,k)=(\mu^*(k_0),k_0)}\cdot\frac{\partial \mu^*(k)}{\partial k}\bigg|_{k=k_0}+\frac{\partial \lambda_1\left({\mu} ,\frac{\alpha F(\tilde{u})-\theta}{d(\tilde{u};k)}\right)}{\partial k}\bigg|_{({\mu} ,k)=(\mu^*(k_0),k_0)}=0.
$$
This combined with Lemma \ref{lem01} (ii) gives $\frac{\partial \mu^*(k)}{\partial k}|_{k=k_0}>0$. Finally, the last part of this lemma is derived directly from Lemma \ref{lem04}.
\end{proof}

Lemma \ref{lem05*} tells us that the sign of quantities defined in \eqref{2022062806} or \eqref{2022062807} determines the monotonicity of $\mu^*(k)$ with respect to $k$. In general these quantities may change signs by as $k$ varies from $0$ to $k^*$. In the following Lemma, we shall show that the sign of quantities defined in \eqref{2022062806} or \eqref{2022062807} can be determined if $m(x)$ is monotonic.
\begin{lemma}\label{lem06*}
Assume $\Omega=[0,L]$, $m_x\geq0$ in $(0,L)$ or $m_x\leq0$ in $(0,L)$ and $d(u)=d(u;k)$, where $d(u;k)=e^{-ku}$ or $(1+u)^{-k}$. If $\theta\in[\theta_0,\alpha F(\tilde{u}_{\max}))$, then  $\frac{\partial \mu^*(k)}{\partial k}>0$ for $k\in[0,k^*)$.
\end{lemma}
\begin{proof}
We only consider the case $m_x\geq0$ in $(0,L)$ and $d(u;k)=e^{-ku}$ while other cases can be proven similarly. Recall that $\tilde{u}$ satisfies
$$
\begin{cases}
\epsilon\tilde{u}_{xx}+\tilde{u}(m(x)-\tilde{u})=0, &\mbox{in}\;(0,L),\\
\tilde{u}_x(0)=\tilde{u}_x(L)=0.
\end{cases}
$$
Define $\eta:=\frac{\tilde{u}_x}{\tilde{u}}$ on $[0,L]$. Then $\eta$ satisfies
$$
\begin{cases}
-\epsilon\eta_{xx}+(\tilde{u}-2\epsilon\eta_x)\eta=m_x\geq0, &\mbox{in}\;(0,L),\\
\eta(0)=\eta(L)=0.
\end{cases}
$$
By the strong maximum principle, one finds that
$$
\eta>0\quad\mbox{in}\;(0,L),
$$
which yields that $\tilde{u}_x>0$ in $(0,L)$. Recall that $\phi_1(\mu^*(k),(\alpha F(\tilde{u})-\theta)e^{k\tilde{u}})$ satisfies
\begin{equation}\label{2022072901}
\begin{cases}
\mu^*(k)\phi_{xx}+\phi(\alpha F(\tilde{u})-\theta)e^{k\tilde{u}}=0, &\mbox{in}\;(0,L),\\
\phi_x(0)=\phi_x(L)=0.
\end{cases}
\end{equation}
Integrating the first equation of \eqref{2022072901} over $(0,L)$, one obtains
$$
\int_0^L\phi_1(\alpha F(\tilde{u})-\theta)e^{k\tilde{u}}\mathrm{d}x=0,
$$
where $\phi_1$ denotes $\phi_1(\mu^*(k),(\alpha F(\tilde{u})-\theta)e^{k\tilde{u}})$ for simplicity.
This fact combined with $\tilde{u}_x>0$ in $(0,L)$, implies that there exist some $x^*\in(0,L)$ such that
$$\mathrm{sgn}(\alpha F(\tilde{u}(x))-\theta)=\mathrm{sgn}(x-x^*).$$
This fact together with $\phi_1>0$ on $[0,L]$ and the first equation of \eqref{2022072901} yields that
$$\mathrm{sgn}((\phi_1)_{xx})=-\mathrm{sgn}(x-x^*).$$
which alongside the boundary conditions $(\phi_1)_x(0)=(\phi_1)_x(L)=0$ indicates that
\begin{equation}\label{2022072902}
(\phi_1)_x>0\quad\mbox{in}\;(0,L).
\end{equation}
%
Combining \eqref{2022072902}, $\tilde{u}_x>0$ in $(0,L)$, \eqref{2022062806} and Lemma \ref{lem05*}, one concludes that $\frac{\partial \mu^*(k)}{\partial k}>0$ for any $k\in[0,k^*)$.
\end{proof}

\section{Existence and non-existence of positive solutions to system \eqref{KO2}}
In this section, we shall prove the existence and non-existence of positive solutions to system~\eqref{KO2} with help of index theory based on the results established in section 2. We start by reviewing some well-known results of the index theory.
\subsection{Index theory}
 Let $E$ be a real Banach space and $W$ be a closed convex set of $E$ such that $W-W$ is dense in $E$. The closed convex set $W$ is said to be a cone if $eW\subset W$ for all $e\geq0$ and $W\cap \{-W\}=0$. Define
\[
W_y\triangleq \overline{\{x\in E|y+ex\in W \mbox{ for some $e>0$}\}}.
\]
Denote the maximal linear subspace of $E$ contained in $W_y$ by $S_y$. Assume that $T:\ E\to E$ is a compact linear and Fr\'{e}chet differentiable operator on $E$ such that $y\in W$ is a fixed point of $T$ and $T(W)\subseteq W$. If there exists a closed linear subspace $E_y$ of $E$ such that $E=S_y\bigoplus E_y$ and $W_y$ is generating, then the following result holds.

\begin{lemma}[\cite{Dancer,Ruan}]\label{lem06}
Let $P: \ E\to E_y$ be the projection operator. Then $\mathrm{index}_W(T,y)$ exists if the Fr\'{e}chet derivative $T'(y)$ of $T$ at $y$ has no non-zero fixed point in $W_y$. Moreover,
\begin{itemize}
\item[(i)]
$\mathrm{index}_W(T,y)=0$ if $P\circ T'(y)$ has an eigenvalue bigger than $1$; Otherwise,
\item[(ii)]
$\mathrm{index}_W(T,y)=\mathrm{index}_{S_y}(T'(y),0)=(-1)^{\imath}$, where ${\mathrm{index}}_{S_y}(T'(y),0)$ is the index of the linear operator $T'(y)$ at $0$ in the space $S_y$ and $\imath$ is the sum of
algebraic multiplicities of the eigenvalues of $T'(y)$ restricted in $S_y$ which are greater than 1.
\end{itemize}
\end{lemma}

\subsection{Preliminary results}
We first  quote  an important result on the eigenvalue problem \cite{Dancer1984,li1988coexistence}.
\begin{lemma}\label{lem05}
Assume $r(x)\in C([0,L])$, $\mu>0$, and $M>0$ such that $M+r>0$ on $\Omega$. If $\lambda_1(\mu,r)>0$, then the weighted eigenvalue problem,
\begin{equation}\label{2022061801}
\begin{cases}
-\mu\Delta \phi+M\phi=\kappa(M+r)\phi, \qquad
&x\in \Omega,
\\
\nabla \phi\cdot \n=0,&x\in \partial\Omega,
\end{cases}
\end{equation}
has an eigenvalue $\kappa$ smaller than $1$. If $\lambda_1(\mu,r)<0$, then it  has no eigenvalue smaller than or equal to $1$.
\end{lemma}

Next we give an upper bound on possible positive solutions of system~\eqref{KO2}.

\begin{lemma}\label{lem07}
Let $(u,w)$ be a positive  solution of system~\eqref{KO2}. Then
\begin{equation}\label{2022061902}
u<\tilde{u}\leq {m}_{\max} \mbox{ and $w\leq c_0$ on $\Omega$, }
\end{equation}
where $m_{\max}=\max\limits_{x\in\bar{\Omega}}m$ and $c_0>0$ is a constant depending on $m$, $\alpha$, $\theta$, $F(\cdot)$ and $d(\cdot)$.
\end{lemma}
\begin{proof}
Combining the standard method of upper-lower solutions and the maximum principle, one can deduce that
$$
u<\tilde{u}\leq {m}_{\max}\;\;\mbox{on}\;\;\Omega.
$$
Multiplying the first equation of system \eqref{KO2} by $\alpha$, adding the resulting equation to the second equation of system \eqref{KO2} and integrating it on $\Omega$, one obtains
$$
\int_\Omega\frac{\theta}{d(u)}w\mathrm{d}x=\alpha\int_\Omega u(m-u)\mathrm{d}x.
$$
This combined with $u<{m}_{\max}$ on $\Omega$ and  $(H_3)$ yields that
$$
\int_\Omega w\mathrm{d}x<\frac{\alpha d(0)}{\theta}\int_\Omega u(m-u)\mathrm{d}x<\frac{\alpha d(0)m_{\max}^2}{4\theta},
$$
which together with  \cite[Theorem 3.1]{Alikakos} implies that there exists $c_0$ depending on $m$, $\alpha$, $\theta$, $F(\cdot)$ and $d(\cdot)$ (independents on $\mu$ and $\epsilon$) such that
$
w\leq c_0\;\;\mbox{on}\;\;\Omega.
$
\end{proof}

Before moving forward,  we introduce some notations.
\begin{eqnarray*}
X&= & \{u\in C^1(\bar{\Omega})\cap C^2(\Omega)\ |\ \nabla u\cdot \n=0\;\mbox{on}\;\partial\Omega\},
\\
E&=&C(\bar{\Omega})\times C(\bar{\Omega}),
\\
W&=&C^{+}(\bar{\Omega})\times C^{+}(\bar{\Omega}), \mbox{ where $C^{+}(\bar{\Omega})=\{u\in C(\bar{\Omega})\ |\ u\geq 0\}$,}
\\
\mathcal{D} & = & \left\{(u,w)\in W\ | \ u<1+m_{\max}, w<1+c_0\right\}.
\end{eqnarray*}
 Let $\mathcal{T}_1^{-1}$ be the inverse operator of $\mathcal{T}_1$ with $\mathcal{T}_1(u)=-\epsilon \Delta u+\tilde{M}u$ for $u\in X$
and $\mathcal{T}_2^{-1}$ be the inverse operator of $\mathcal{T}_2$ with $\mathcal{T}_2(w)=-\mu\Delta w+\tilde{M}w$ for $u\in X$.
For any $\delta\in[0,1]$, we define $T_\delta:\mathcal{D}\rightarrow W$ by
\[
T_\delta(u,w)=
\begin {pmatrix}
\mathcal{T}_1^{-1}\left[u\left(\tilde{M}+\delta m(x)
-u-\frac{F(u)w}{d(u)u}\right)\right]
\\[6pt]	
\mathcal{T}_2^{-1}\left[w\left(\tilde{M}+\frac{\alpha F(u)-\theta}{d(u)}\right)\right]
\end {pmatrix}, \qquad (u,w)\in \mathcal{D},
\]
where $\tilde{M}$ is large such that
$$
\tilde{M}-|m(x)|
-u-\frac{F(u)w}{d(u)u}>0\quad\mbox{and}\quad\tilde{M}+\frac{\alpha F(u)-\theta}{d(u)}>0,\quad\mbox{for}\;\;(u,w)\in\mathcal{D}.
$$
For example, one can choose $\tilde{M}=\frac{\theta}{d(m_{\max})}+2\|m\|_{L^\infty}+\frac{c_0}{d(m_{\max})}\max\limits_{u\in[0,m_{\max}]}\frac{F(u)}{u}$, where $\max\limits_{u\in[0,m_{\max}]}\frac{F(u)}{u}$ is bounded due to the assumption $(H_2)$. It is well-known that $T_1$ is a compact operator and $T_1(\mathcal{D})\subseteq W$. Clearly system~\eqref{KO2} has a positive solution if and only if  $T_1$ admits a positive fixed point on $\mathcal{D}$ by Lemma~\ref{lem07}.

Direct computations yield
\begin{eqnarray*}
& W_{(0,0)}=C^{+}(\bar{\Omega})\times C^{+}(\bar{\Omega}),\;S_{(0,0)}=\{(0,0)\},\;E_{(0,0)}=E, &
\\
& W_{(\tilde{u},0)}=C(\bar{\Omega})\times C^{+}(\bar{\Omega}),\;S_{(\tilde{u},0)}=C(\bar{\Omega})\times\{0\},\;E_{(\tilde{u},0)}=\{0\}\times C(\bar{\Omega}).&
\end{eqnarray*}

With the above preparations, we start to calculate the {\it index} of $(0,0)$ and $(\tilde{u},0)$.
\begin{lemma}\label{lem08}
The following results on the index holds.
\begin{itemize}
\item[(i)] $\mathrm{index}_W(T_1,(0,0))=0$.

\item[(ii)] $\mathrm{index}_W(T_1,(\tilde{u},0))=
\begin{cases}
1,&\mbox{if $(\tilde{u},0)$ is linearly stable}, \\
0,&\mbox{if $(\tilde{u},0)$ is linearly unstable}.
\end{cases}
$

\item[(iii)] ${\rm{\rm{deg}}}_W(I-T_1,\mathcal{D})=1$.
\end{itemize}
\end{lemma}
\begin{proof}
For statement (i), we linearize $T_1$ at $(0,0)$ to obtain
\begin{equation}\nonumber
DT_1(0,0)(\phi,\psi)=
\begin{pmatrix}
\mathcal{T}_1^{-1}[(\tilde{M}+m)\phi]
\\
\mathcal{T}_2^{-1}\left[\left(\tilde{M}-\frac{\theta}{d(0)}\right)\psi\right]
\end{pmatrix}.
\end{equation}
It is straightforward to see that  the operator $DT_1(0,0)$ has no non-zero fixed point in $W_{(0,0)}$ due to the fact that $\lambda_1(\epsilon,m)>0$ and $\lambda_1\left(\mu,-\frac{\theta}{d(0)}\right)=-\frac{\theta}{d(0)}<0$. From Lemma \ref{lem03} and  Lemma \ref{lem05}, it follows that $DT_1(0,0)$ admits an eigenvalue bigger than 1 with corresponding eigenfunction $(\phi_1(\epsilon,m),0)$. Therefore, by Lemma~\ref{lem06}, we get $\mathrm{index}_W(T_1,(0,0))=0$.

As to assertion (ii), linearizing $T_1$ at $(\tilde{u},0)$, one has
\[
DT_1(\tilde{u},0)(\phi,\psi)=
\begin {pmatrix}
\mathcal{T}_1^{-1}\left[\phi\left(\tilde{M}+m(x)
-2\tilde{u}\right)-\frac{F(\tilde{u})}{d(\tilde{u})}\psi\right]
\\[6pt]	
\mathcal{T}_2^{-1}\left[\psi\left(\tilde{M}+\frac{\alpha F(\tilde{u})-\theta}{d(\tilde{u})}\right)\right]
\end {pmatrix}.
\]
We will show that the operator $DT_1(\tilde{u},0)$ has no non-zero fixed point in $W_{(\tilde{u},0)}$. Otherwise, assume that $DT_1(\tilde{u},0)$ has a non-zero fixed point $(\phi,\psi)$ in $W_{(\tilde{u},0)}$. Then $(\phi,\psi)$ satisfies
\[
\begin{cases}
\epsilon\Delta \phi+(m-2\tilde{u})\phi-\frac{F(\tilde{u})}{d(\tilde{u})}\psi=0, \quad &\mbox{in}\; \Omega,
\\
\mu\Delta \psi+\frac{\alpha F(\tilde{u})-\theta}{d(\tilde{u})}\psi=0, \quad &\mbox{in}\;\Omega,
\\
\nabla\phi\cdot \n=\nabla\psi\cdot \n=0,\quad&\mbox{on}\;\partial\Omega.
\end{cases}
\]
If $\psi=0$, then $\phi=0$ due to $\lambda_1(\epsilon,m-2\tilde{u})<\lambda_1(\epsilon,m-\tilde{u})=0$ by Lemma~\ref{lem01}(iii). Thus, one obtains $\psi\in C^{+}(\bar{\Omega})\setminus \{0\}$, which further implies that $\lambda_1\big(\mu,\frac{\alpha F(\tilde{u})-\theta}{d(\tilde{u})}\big)=0$.
This contradicts our assumption that $(\tilde{u},0)$ is linearly stable ($\lambda_1\big(\mu,\frac{\alpha F(\tilde{u})-\theta}{d(\tilde{u})}\big)<0$)  or unstable ($\lambda_1\big(\mu,\frac{\alpha F(\tilde{u})-\theta}{d(\tilde{u})}\big)>0$).
Hence, the operator $DT(\tilde{u},0)$ does not have non-zero fixed points in $W_{(\tilde{u},0)}$.  If $(\tilde{u},0)$ is linearly unstable, that is,  $\lambda_1\big(\mu,\frac{\alpha F(\tilde{u})-\theta}{d(\tilde{u})}\big)>0$, one attains that $\mathcal{T}_2^{-1}\big[\cdot\big(\tilde{M}+\frac{\alpha F(\tilde{u})-\theta}{d(\tilde{u})}\big)\big]$ has an eigenvalue bigger than 1 by Lemma \ref{lem05}. This combined with Lemma \ref{lem06} gives that $\mathrm{index}_W(T_1,(\tilde{u},0))=0$. On the other hand, if  $(\tilde{u},0)$ is linearly stable, by Lemma \ref{lem05}, one knows that all the eigenvalues of the operator $\mathcal{T}_2^{-1}\big[\cdot\big(\tilde{M}+\frac{\alpha F(\tilde{u})-\theta}{d(\tilde{u})}\big)\big]$ are smaller than 1. This together with Lemma \ref{lem06} yields that
$$
\mathrm{index}_W(T_1,(\tilde{u},0))=(-1)^{\imath},
$$
where $\imath$ is the sum of
algebraic multiplicities of the eigenvalues of the operator $DT_1(\tilde{u},0)$ restricted in $S_y$ which are greater than 1.

Next we show that all the eigenvalues of the operator $DT_1(\tilde{u},0)$  restricted in $S_y$  are smaller than 1. If not, we assume the operator $DT_1(\tilde{u},0)$ admits an eigenvalue $\kappa_0\geq1$ with  eigenfunction $(\phi,0)\in S_y$ satisfying $\|\phi\|_{L^2}=1$. Then $\kappa_0$ and $(\phi,0)$ satisfy
\[
\begin{cases}
-\epsilon\Delta \phi+\tilde{M}\phi=\frac{\phi}{\kappa_0}(\tilde{M}+m-2\tilde{u}), \quad &\mbox{in}\; \Omega,
\\
\nabla\phi\cdot \n=0,\quad&\mbox{on}\;\partial\Omega.
\end{cases}
\]
This contradicts the fact that $\lambda_1(\epsilon,m-2\tilde{u})<0$ and Lemma \ref{lem05}. Therefore, one concludes
$$
\mathrm{index}_W(T_1,(\tilde{u},0)=(-1)^0=1,\quad\mbox{when}\;(\tilde{u},0)\;\mbox{is linearly stable}.
$$

Finally, we prove that ${\rm{\rm{deg}}}_W(I-T_1,\mathcal{D})=1$.
If $T_\delta$ has a fixed point $(u,w)$, then it satisfies
\begin{equation}\label{2022061901}
\begin{cases}
\epsilon\Delta u+u(\delta m(x)-u)-\frac{F(u)w}{d(u)}=0, &\mbox{in}\;\Omega,\\
\mu\Delta w+\alpha \frac{F(u)w}{d(u)}-\theta \frac{w}{d(u)}=0, &\mbox{in}\;\Omega,\\
\nabla u \cdot \n=\nabla w \cdot \n=0,&\mbox{on}\;\partial\Omega.
\end{cases}
\end{equation}
Similar to  Lemma~\ref{lem07}, for all  $\delta\in[0,1]$, one can show that if system~\eqref{2022061901} has a positive solution $(u,w)$ then it satisfies \eqref{2022061902} (if necessary, one can choose large $c_0$). Then, $T_\delta$ doesn't have any fixed point on $\partial \mathcal{D}$. Thus, by the homotopy invariance, one obtains
\begin{equation}\label{2022061903}
\deg_W(I-T_1,\mathcal{D})=\deg_W(I-T_{\delta},\mathcal{D}).
\end{equation}
Obviously, system \eqref{2022061901} only admits non-negative solution $(0,0)$ and $(\tilde{u}_{\delta m},0)$ ($\tilde{u}_{\delta m}$ denotes the unique positive solution of \eqref{2022062101} by replacing $m$ with $\delta m$) when $\delta$ is small enough. Therefore, one has
\begin{equation}\label{2022061904}
\deg_W(I-T_{\delta},\mathcal{D})=\mathrm{index}_W(T_{\delta},(0,0))+\mathrm{index}_W(T_{\delta},(\tilde{u}_{\delta m},0)),
\end{equation}
where $\delta$ is small enough.
Linearizing $T_{\delta}$ at $(0,0)$, one gets
\begin{equation}\nonumber
DT_{\delta}(0,0)(\phi,\psi)=
\begin{pmatrix}
\mathcal{T}_1^{-1}[(\tilde{M}+\delta m)\phi]
\\
\mathcal{T}_2^{-1}\left[\left(\tilde{M}-\frac{\theta}{d(0)}\right)\psi\right]
\end{pmatrix}.
\end{equation}
Since $\lambda_1(\epsilon,\delta m)>0$ and $\lambda_1\left(\mu,-\frac{\theta}{d(0)}\right)=-\frac{\theta}{d(0)}<0$, similar to the results in (i), by Lemma \ref{lem05} and Lemma \ref{lem06}, we have
\begin{equation}\label{2022091701}
\mathrm{index}_W(T_{\delta},(0,0))=0,
\end{equation}
where $\delta$ is small enough. It is easy to derive that $(\tilde{u}_{\delta m},0)$ is linearly stable when $\delta$ is small enough. Therefore, from statement (ii), it follows that
$$
\mathrm{index}_W(T_{\delta},(\tilde{u}_{\delta m},0))=1,\quad \mbox{when $\delta$ is small enough},
$$
which along with \eqref{2022061903}, \eqref{2022061904} and  \eqref{2022091701} completes the proof.
\end{proof}

With the help of Lemma \ref{lem08}, we give the sufficient conditions for the existence of positive solution to system \eqref{KO2}.
\begin{lemma}\label{lem09}
If $(\tilde{u},0)$ is linearly unstable, then system  \eqref{KO2} admits at least  one positive solution.
\end{lemma}
\begin{proof}
If system \eqref{KO2} doesn't have any positive solution, by the additivity of indices and Lemma \ref{lem08}, we have
$$
1={\rm{\rm{deg}}}_W(I-T_1,\mathcal{D})=\mathrm{index}_W(T_1,(0,0))+\mathrm{index}_W(T_1,(\tilde{u},0))=0+0=0,
$$
which is impossible. Hence, system  \eqref{KO2} admits at least one positive solution when $(\tilde{u},0)$ is linearly unstable.
\end{proof}

\subsection{Main results}Now it is in a position to state our main results on the existence and non-existence of positive solutions to \eqref{KO2}.
\begin{theorem}\label{thm01}
Given $\epsilon,\alpha>0$, assume $(H_1)$, $(H_2)$ and $(H_3)$ hold. Let $\tilde{u}$ be the unique solution of \eqref{2022062101} and $\theta_0=\frac{\alpha}{|\Omega|}\int_\Omega F(\tilde{u})\mathrm{d}x.$ Then the following results hold.
\noindent \begin{itemize}
\item[(i)]
If $\theta\in[0,\alpha F(\tilde{u}_{\min})]$, then system \eqref{KO2} admits at least  a positive solution.

\item[(ii)]
If $\theta\in[\alpha F(\tilde{u}_{\max}),\infty)$, then system \eqref{KO2} doesn't admit any positive solution.

\item[(iii)] If $\theta\in (\alpha F(\tilde{u}_{\min}), \alpha F(\tilde{u}_{\max}))$ and  $d(u)= d(u;k)$ where $d(u;k)$ is given in \eqref{du}, then the following results follow.
\begin{itemize}
\item[(a)]Fixing all the parameters except $\mu$, if $\theta\in\left(\alpha F(\tilde{u}_{\min}),\theta_k\right]$ with $\theta_k$ defined in \eqref{thetak}, then system \eqref{KO2} admits a positive solution for any $\mu>0$; while if $\theta\in\left(\theta_k,\alpha F(\tilde{u}_{\max})\right)$, then there exits some $\mu^*(\theta)>0$ satisfying $\textstyle \lambda_1\big(\mu^*,\frac{\alpha F(\tilde{u})-\theta}{d(\tilde{u};k)}\big)=0$ such that system \eqref{KO2} admits a positive solution for all ${\mu} \in (0,\mu^*(\theta))$ and the semi-trivial solution $(\tilde{u}, 0)$ is linearly stable  for all ${\mu}>\mu^*(\theta)$.
\item[(b)]Fixing all the parameters except $\mu$ and $k$, we have
\begin{itemize}
\item[(b.1)] If $\theta\in\left(\alpha F(\tilde{u}_{\min}),\theta_0\right]$, then system \eqref{KO2} admits at least  a positive solution for any ${\mu} >0$ and $k\geq0$.
\item[(b.2)]If $\theta\in\left(\theta_0,\alpha F(\tilde{u}_{\max})\right)$, then there exist $k^*(\theta)>0$
satisfying \eqref{2022062802} such that:
\begin{itemize}
\item[(b.2A)] If  $k\geq k^*(\theta)$, then system \eqref{KO2} admits a positive solution for any ${\mu} >0$;
\item[(b.2B)] If  $k\in[0,k^*(\theta))$, then there exists $\mu^*(k)$ satisfying $\lambda_1\big(\mu^*(k),\frac{\alpha F(\tilde{u})-\theta}{d(\tilde{u})}\big)=0$ such that system \eqref{KO2}  admits a positive solution for all ${\mu} \in (0,\mu^*(k))$ and $(\tilde{u}, 0)$ is linearly stable for all ${\mu} >\mu^*(k)$. Furthermore there exist a constant $\tilde{k}(\theta)>0$ so that system \eqref{KO2} has no positive solutions for all $k\in[0,\tilde{k}(\theta)]$ and  ${\mu}>\mu^*(k)$.
\end{itemize}
\end{itemize}
\end{itemize}
\end{itemize}
\end{theorem}
\begin{proof}
The results stated in assertions (i), (iii)-(a), and (iii)-(b.1) follow directly from Lemma \ref{lem04} and Lemma \ref{lem09}. For the assertion (ii), we use a contradictive argument by assuming that system \eqref{KO2} admits a positive solution $(u,w)$. From the second equation of system \eqref{KO2}  and the Krein-Rutman Theorem \cite{KreinRutman1948}, it follows that
\begin{equation}\label{2022063001}
\lambda_1\left(\mu,\frac{\alpha F(u)-\theta}{d(u)}\right)=0.
\end{equation}
On the other hand, due to assumptions $\theta\geq \alpha F(\tilde{u}_{\max})$, $(H_1)$ and $(H_2)$, one concludes that
$$
\alpha F(u)-\theta\leq,\not\equiv0\;\;\mbox{in}\;\Omega,
$$
which together with  $\lambda_1(\mu,0)=0$ and Lemma \ref{lem01}-(iii) implies that
$
\lambda_1\big(\mu,\frac{\alpha F(u)-\theta}{d(u)}\big)<0.
$
This contradicts \eqref{2022063001}. Therefore, the results in statement (ii) holds.

Finally we prove the results in the assertion (iii)-(b.2). First the result (b.2A) in the statement (b.2) come from Lemma \ref{lem04} and Lemma \ref{lem09} directly. It remains only to show result (B) in the statement (b.2). Given $k\in[0,k^*(\theta))$, using Lemma \ref{lem01} and $\theta\in (\alpha F(\tilde{u}_{\min}), \alpha F(\tilde{u}_{\max}))$, we have
$$
\lim\limits_{{\mu} \to0}\lambda_1\left({\mu},\frac{\alpha F(\tilde{u})-\theta}{d(\tilde{u};k)}\right)=\max\limits_{x\in\bar{\Omega}}\frac{\alpha F(\tilde{u})-\theta}{d(\tilde{u};k)}>0
$$
and
$$
\lim\limits_{{\mu} \to\infty}\lambda_1\left({\mu},\frac{\alpha F(\tilde{u})-\theta}{d(\tilde{u};k)}\right)=
\frac{\int_\Omega\frac{\alpha F(\tilde{u})-\theta}{d(\tilde{u};k)}\mathrm{d}x}{|\Omega|}<0,\ \mbox{due to $k<k^*(\theta)$} \ \mbox{and} \ \eqref{2022062802}.
$$
These facts combined with Lemma \ref{lem01} (ii) and Lemma \ref{lem09} imply the first part of (B) in (b.2). Next proceed to prove the existence of $\tilde{k}(\theta)$.
To this end, we consider the case $d(u;k)=e^{-ku}$ only and the other case $d(u;k)=(1+u)^{-k}$ can be treated similarly. We define $\mathfrak{h}(x):=(\alpha F(x)-\theta)e^{kx}$, $x\in[0,\tilde{u}_{\max}]$. Then, one has
$$
\mathfrak{h}_x(x)=[\alpha F_x(x)+k(\alpha F(x)-\theta)]e^{kx},
$$
which combined with assumption $(H_2)$ implies that there exists $\tilde{k}>0$, such that
\begin{equation}\label{2022070701}
\mathfrak{h}_x(x)>0,\quad x\in[0,\tilde{u}_{\max}].
\end{equation}
Assume $k\leq \tilde{k}$ and ${\mu} >\mu^*(k)$, we will show that system \eqref{KO2} doesn't admit any positive solution. By contradiction, assume that system \eqref{KO2} admits a positive solution $(u,w)$. From the second equation of \eqref{KO2},
it follows that
$$
\lambda_1({\mu} ,(\alpha F(u)-\theta)e^{ku})=0.
$$
This together with \eqref{2022070701}, Lemma \ref{lem01}, and Lemma \ref{lem07} yields that
$
\lambda_1({\mu} ,(\alpha F(\tilde{u})-\theta)e^{k\tilde{u}})>0,
$
which alongside Lemma \ref{lem01} indicates that
$$
\lambda_1(\mu^*(k),(\alpha F(\tilde{u})-\theta)e^{k\tilde{u}})>0.
$$
This contradicts the definition of $\mu^*(k)$, that is, $\lambda_1(\mu^*(k),(\alpha F(\tilde{u})-\theta)e^{k\tilde{u}})=0$. So, system \eqref{KO2} doesn't admit any positive solution.
\end{proof}


As a direct consequence of Theorem \ref{thm01}, we have the following results for the predator-prey system with random dispersal.
\begin{corollary}\label{thmkzero}
Given $\epsilon,\alpha>0$, assume $d(u)=1$, $(H_1)$ and $(H_2)$ hold. Then the following results hold.
\begin{itemize}
\item[(i)]
If $\theta\in[0,\alpha F(\tilde{u}_{\min})]$, then system \eqref{KO2} admits at least  a positive solution.
\item[(ii)]
If $\theta\in[\alpha F(\tilde{u}_{\max}),\infty)$, then system \eqref{KO2} doesn't admit any positive solution.
\item[(iii)] If $\theta\in (\alpha F(\tilde{u}_{\min}), \alpha F(\tilde{u}_{\max}))$, the following results hold true.

\begin{itemize}
\item[(a)]If $\theta\in\left(\alpha F(\tilde{u}_{\min}),\theta_0\right]$, then system \eqref{KO2} admits at least  a positive solution for any ${\mu} >0$.
\item[(b)]If $\theta\in\left(\theta_0,\alpha F(\tilde{u}_{\max})\right)$, then there exists a constant $\mu^*$  satisfying $\lambda_1(\mu^*,\alpha F(\tilde{u})-\theta))=0$  such that system \eqref{KO2} doesn't admit any positive solution for $\mu>\mu^*$ while admits a positive solution for all $\mu\in (0,\mu^*)$.
\end{itemize}
\end{itemize}
\end{corollary}

\begin{remark}{\em We have several remarks in connection with the results of Theorem \ref{thm01}.
\begin{itemize}
\item  Comparing the results of Theorem \ref{thm01} with those of Corollary \ref{thmkzero}, we see that the density-dependent dispersal will have no impact on the species coexistence when the predator's death rate $\theta>0$ is small (i.e. $\theta\leq\alpha F(\tilde{u}_{\min})$) or large (i.e. $\theta\geq\alpha F(\tilde{u}_{\max})$). However if the value of $\theta$ is moderate (i.e. $\theta\in (\alpha F(\tilde{u}_{\min}), \alpha F(\tilde{u}_{\max}))$), the density-dependent dispersal will have evident impact on the species coexistence. Considering the case $d(u)=e^{-ku}$ or $(1+u)^{-k}$ with $k\geq 0$, the results stated in (iii)-(a) of Theorem \ref{thm01} can be illustrated in Figure \ref{fig1}(a) and Figure \ref{fig1}(b) where we see that the parameter regions of $(\theta, \mu)$ for the existence of positive solutions (i.e. $\lambda_1>0$) increases as $k$ increases. This implies that the density-dependent dispersal will increase the chance of species coexistence. The result in  (iii)-(b) of Theorem \ref{thm01} gives another way of understanding the impact of density-dependent dispersal, where for given $\theta\in(\theta_0,\alpha F(\tilde{u}_{\max}))$ coexistence (positive) solutions exist only if $0<\mu<\mu^*(0)$ when $k=0$ (see (iii)-(b) of Corollary \ref{thmkzero}) while exit for any $\mu>0$  when $k>k^*(\theta)$ (see (iii)-(b.2) in Theorem \ref{thm01}), as illustrated in Figure \ref{fig1}(c). For $k\in [0, k^*(\theta))$ and in dimension one, we have shown that $\mu^*(k)$ increases with respect to $k$ (see Lemma \ref{lem06*}) and hence the range of $\mu$ for the coexistence (i.e. $\lambda_1>0$) increases as $k$ increases (see Figure \ref{fig1}(c)). This is also verified by our numerical simulations shown in Figure \ref{fig2}.
    \begin{figure}[t]
\centering
\includegraphics[width=16cm, height=6cm]{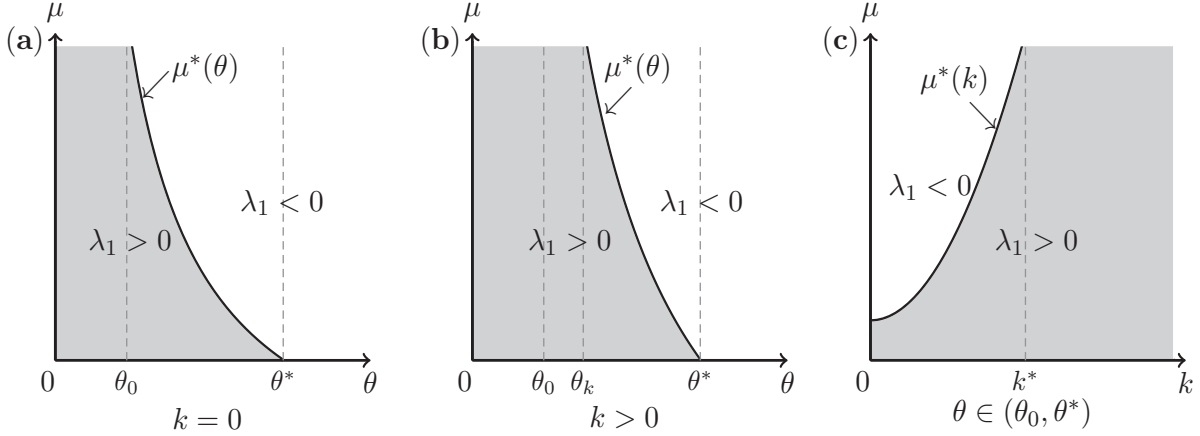}

\caption{Illustration of parameter regimes (shaded regions) for existence of positive solutions to \eqref{KO2} (i.e. $\lambda_1>0$), where $\theta^*=\alpha F(\tilde{u}_{\max})$ and $\theta_k=\frac{\alpha \int_\Omega\frac{ F(\tilde{u})}{d(\tilde{u};k)}\mathrm{d}x}{\int_\Omega\frac{1}{d(\tilde{u};k)}\mathrm{d}x}$ with $d(u;k)=e^{-ku}$ or $d(u;k)=(1+u)^{-k}, k\geq 0$.}
\label{fig1}
\end{figure}
\begin{figure}[h]
\centering
\includegraphics[width=7cm, height=5.5cm]{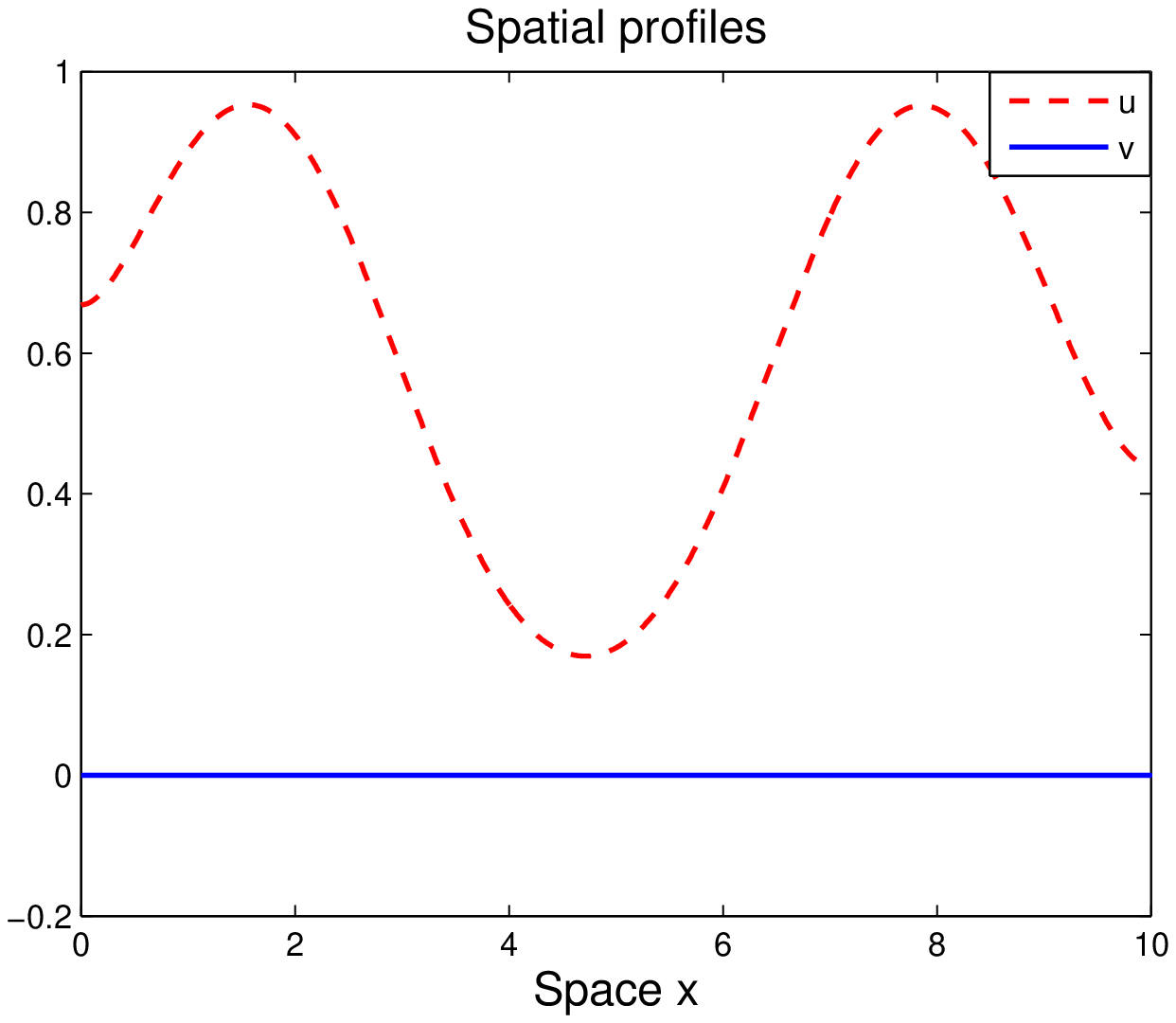}
\includegraphics[width=7cm, height=5.5cm]{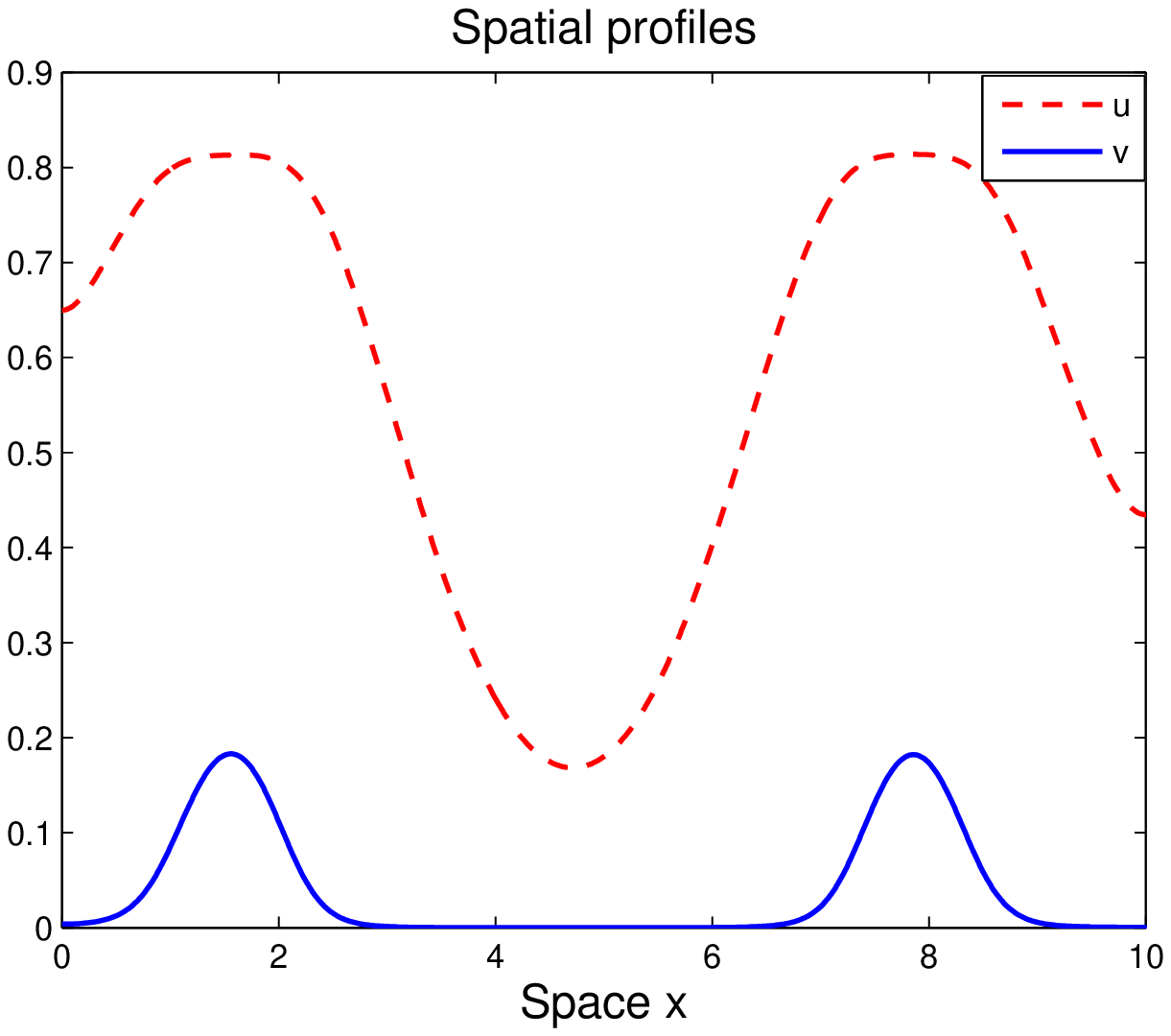}

\hspace{1cm} (a) $d(u)=1$ \hspace{4.8cm} (b) $d(u)=e^{-8u}$

\caption{Numerical simulations of steady state profiles of system \eqref{KO} with random dispersal shown in (a) and density-dependent dispersal shown in (b), where $\epsilon=0.1, \mu=10,  \alpha=1, \theta=0.8, F(u)=u, m(x)=0.5+0.5\sin(x)$.
 The dispersal rate function $d(u)$ is chosen as indicated in the Figure.
}
\label{fig2}
\end{figure}
\item The constant $\tilde{k}$ in (b.2B) of Theorem \ref{thm01} can be explicitly determined for specific1 $F(u)$. For instance, we can choose
$$
\tilde{k}=\begin{cases}
\frac{\alpha}{\theta},&\mbox{if}\;F(u)=u,\\
\frac{\alpha}{\theta(1+\tilde{u}_{max})^2},&\mbox{if}\;F(u)=\frac{u}{1+u}.
\end{cases}
$$
\end{itemize}
}
\end{remark}

\section{Uniqueness and asymptotic profiles}
In this section, we are devoted to investigating the uniqueness and asymptotic profiles of solutions of \eqref{KO2}  as $\varepsilon \to \infty$ (fast prey diffuse) as well as $\mu \to \infty/0$ (fast/slow predator diffusion).

\subsection{Fast prey diffusion}
\begin{theorem}\label{thm5}
Suppose that $(H_1)$, $(H_2)$ and $(H_3)$ hold. Let $\bar{m}=\frac{1}{|\Omega|}\int_\Omega m(x)\mathrm{d}x$. If $\epsilon>0$ is sufficiently large, then  the following results hold.
\begin{itemize}
\item[(i)]
If $\theta>\alpha F(\bar{m})$, then system \eqref{KO2} doesn't have any positive solution;
\item[(ii)]
If $0<\theta<\alpha F(\bar{m})$,  then system \eqref{KO2} admits a unique positive solution.
\end{itemize}
\end{theorem}
\begin{proof}
For the assertion (i), arguing by contradiction, we suppose that system \eqref{KO2} admits a positive solution $(u^i,w^i)$ with $\epsilon=\epsilon^i$, where $\epsilon^i\to+\infty$ as $i\to\infty$. Then $\epsilon^i$ and $(u^i,w^i)$ satisfy
\begin{equation}\label{2022090301}
\begin{cases}
\epsilon^i\Delta u^i+u^i(m(x)-u^i)-\frac{F(u^i)}{d(u^i)}w^i=0, &\mbox{in}\;\Omega,\\
\mu\Delta w^i+\frac{\alpha F(u^i)-\theta}{d(u^i)}w^i=0, &\mbox{in}\;\Omega,\\
\nabla u^i \cdot \n=\nabla w^i \cdot \n=0,&\mbox{on}\;\partial\Omega.
\end{cases}
\end{equation}
Using Lemma \ref{lem07}, for any $i\geq 1$, one has
$$
\|u^i\|_{L^{\infty}(\Omega)}\leq m_{\max}\;\;
\mbox{and}\;\;
\|w^i\|_{L^{\infty}(\Omega)}\leq c_0.
$$
Applying the elliptic regularity (cf. \cite{GilbargTrudinger2001}), we have $\|u^i\|_{W^{2, p}(\Omega)}$ and $\|w^i\|_{W^{2, p}(\Omega)}$ are uniformly bounded for any $1< p<\infty$ and $i\geq1$.
By the Sobolev imbedding theorem, one can deduce from (\ref{2022090301}) that $(u^i,w^i)$, passing to a subsequence if necessary, converges to some nonnegative function $(u^\infty,w^\infty)$ in  $C^1(\Omega)$ as $i\to\infty$, where $(u^\infty,w^\infty)$ satisfies (in the weak sense)
\begin{equation}\label{2022090302}
\begin{cases}
\Delta u_{\infty}=0, &\mbox{in}\;\Omega,\\
\mu\Delta w_{\infty}+ w_{\infty}\frac{\alpha F(u_{\infty})-\theta}{d(u_{\infty})}=0, &\mbox{in}\;\Omega,\\
\nabla u^\infty \cdot \n=\nabla w^\infty \cdot \n=0,&\mbox{on}\;\partial\Omega.
\end{cases}
\end{equation}
From Proposition \ref{prop2} (ii), Lemma \ref{lem07}, $(H_2)$, and the assumption $\theta>\alpha F(\bar{m})$, it follows that
$$
\alpha F(u^\infty)-\theta\leq \alpha F(\bar{m})-\theta<0,
$$
which indicates that
\begin{equation}\label{2022090303}
\alpha F(u^i)-\theta<0,\quad\mbox{when $i$ is large}.
\end{equation}
Integrating the second equation of \eqref{2022090301} on $\Omega$, one obtains
$$
\int_\Omega\frac{\alpha F(u^i)-\theta}{d(u^i)}w^i\mathrm{d}x=0,
$$
which contradicts \eqref{2022090303} when $i$ is large. Hence the results in assertion (i) are obtained.

For assertion (ii), by Proposition \ref{prop2} (ii) and Lemma \ref{lem01} (i), one has
$$
\lim\limits_{\epsilon\to\infty}\lambda_1\left(\mu,\frac{\alpha F(\tilde{u}_\epsilon)-\theta}{d(\tilde{u}_\epsilon)}\right)=\lambda_1\left(\mu,\frac{\alpha F(\bar{m})-\theta}{d(\bar{m})}\right)=\frac{\alpha F(\bar{m})-\theta}{d(\bar{m})}>0,
$$
where $\tilde{u}_\epsilon$ denotes the unique positive solution of \eqref{2022062101}.
This combined with Lemma \ref{lem09} implies that, there exits some large $\epsilon^*$ such that system  \eqref{KO2} admits at least  one positive solution for $\epsilon\geq \epsilon^*$. For $\epsilon^i\geq \epsilon^*$ satisfying  $\epsilon^i\to+\infty$ as $i\to\infty$, we will prove that any positive solution $(u^i,w^i)$ of system \eqref{KO2} with $\epsilon=\epsilon^i$ satisfies that
\begin{equation}\label{2022090304}
\mbox{$(u^i,w^i)$  converge to $(c,w_c)$ in $C^1(\Omega)$ as $i\to\infty$},
\end{equation}
where $c=F^{-1}(\theta/\alpha)$ and $w_c=\frac{c\alpha d(c)}{\theta|\Omega|}\int_\Omega(m-c)\mathrm{d}x$. Here $F^{-1}(\cdot)$ denotes the inverse of $F(\cdot)$. Following the approach as that in the proof of assertion (i), it suffices to show that $(u^\infty,w^\infty)=(c,w_c)$. From the first equation of \eqref{2022090302}, it follows that  $u^\infty=c_1$ for some constant $c_1>0$. Let $\hat{w}^i:=\frac{w^i}{\|w^i\|_{L^\infty}}$. Then, $\hat{w}^i$ satisfies
\begin{equation}\label{2022090305}
\begin{cases}
\mu\Delta \hat{w}^i+\frac{\alpha F(u^i)-\theta}{d(u^i)}\hat{w}^i=0, &\mbox{in}\;\Omega,\\
\nabla \hat{w}^i \cdot \n=0,&\mbox{on}\;\partial\Omega.
\end{cases}
\end{equation}
We may assume that $\hat{w}^i\to \hat{w}^{\infty}$ in $C^1(\Omega)$ as $i\to\infty$ (passing
to a subsequence if necessary). Integrating the second equation of \eqref{2022090305} on $\Omega$ and letting $i\to\infty$, we have
$$
\int_\Omega\frac{\alpha F(c_1)-\theta}{d(c_1)}\hat{w}^{\infty}\mathrm{d}x=0,
$$
which along with the facts $\hat{w}^{\infty}\geq0$ and $\|\hat{w}^{\infty}\|_{L^\infty}=1$ implies that $c_1=F^{-1}(\theta/\alpha).$  This together with \eqref{2022090302} yields that
\begin{equation}\label{2022090306}
w^\infty= c_2\geq0.
\end{equation}
Integrating the first equation of \eqref{2022090305} on $\Omega$ and letting $i\to\infty$, one has
$$
\int_\Omega \bigg(c(m-c)-\frac{F(c)}{d(c)}w^\infty\bigg)\mathrm{d}x=0,
$$
which combined with  \eqref{2022090306} indicates that $w^\infty=\frac{c\alpha d(c)}{\theta|\Omega|}\int_\Omega(m-c)\mathrm{d}x$. Hence, \eqref{2022090304} holds.

Define $\mathcal{L}_1: \mathbb{R}\times \bar{H}_0^{2}(\Omega)\times H_n^2(\Omega)\times [0,+\infty)\to \mathbb{R}\times \bar{L}^2(\Omega)\times L^2(\Omega)$  by
$$
\begin{aligned}
&\mathcal{L}_1(\xi,\zeta,w,\gamma)=\\
&
\begin{pmatrix}
\frac{1}{|\Omega|}\int_\Omega\left((\xi+\zeta)(m(x)-\xi-\zeta)-\frac{F(\xi+\zeta)}{d(\xi+\zeta)}w\right)\mathrm{d}x\\[6pt]	\Delta \zeta+\gamma\left[(\xi+\zeta)(m(x)-\xi-\zeta)-\frac{F(\xi+\zeta)}{d(\xi+\zeta)}w-\frac{1}{|\Omega|}\int_\Omega\left((\xi+\zeta)(m(x)-\xi-\zeta)-\frac{F(\xi+\zeta)}{d(\xi+\zeta)}w\right)\mathrm{d}x\right]\\
\mu\Delta w+\frac{\alpha F(\xi+\zeta)-\theta}{d(\xi+\zeta)}w
\end{pmatrix},
\end{aligned}
$$
where $H_n^2(\Omega)=\{u\in H^2(\Omega)|\nabla u\cdot \n=0\;\mbox{on}\;\partial\Omega\}$, $\bar{H}_0^2(\Omega)=\{u\in H_n^2(\Omega)|\int_\Omega u\mathrm{d}x=0\}$, and $\bar{L}^2(\Omega)=\{u\in L^2(\Omega)|\int_\Omega u\mathrm{d}x=0\}$.
Then, one has
$$
\begin{aligned}
&D_{(\xi,\zeta,w)}{\mathcal{L}_1}|_{(\xi,\zeta,w,\gamma)=(c,0,w_c,0)}(\phi,\psi,\eta)\\
&\quad\quad=
\begin{pmatrix}
\frac{1}{|\Omega|}\int_\Omega\left((m-2c)(\phi+\psi)-\frac{F(c)}{d(c)}\eta-\frac{F'(c)d(c)-d'(c)F(c)}{d^2(c)}w_c(\phi+\psi)\right)\mathrm{d}x\\[6pt]	\Delta \psi\\
\mu\Delta\eta+\frac{\alpha F'(c)}{d(c)}w_c(\phi+\psi)
\end{pmatrix},
\end{aligned}
$$
where $c=F^{-1}(\theta/\alpha)$ and $w_c=\frac{c\alpha d(c)}{\theta|\Omega|}\int_\Omega(m-c)\mathrm{d}x$.

{\bf Claim}: $D_{(\xi,\zeta,w)}{\mathcal{L}_1}|_{(\xi,\zeta,w,\gamma)=(c,0,w_c,0)}$ is non-degenerate. To show this, it amounts to show that problem
\begin{equation}\label{2022112901*}
\begin{cases}
\frac{1}{|\Omega|}\int_\Omega\left((m-2c)(\phi+\psi)-\frac{F(c)}{d(c)}\eta-\frac{F'(c)d(c)-d'(c)F'(c)}{d^2(c)}w_c(\phi+\psi)\right)\mathrm{d}x=0,\\
\Delta \psi=0,&\mbox{in}\;\Omega,\\
\mu\Delta\eta+\frac{\alpha F(c)-\theta}{d(c)}\eta+\frac{\alpha F'(c)d(c)-(\alpha F(c)-\theta)d'(c)}{d^2(c)}w_c(\phi+\psi)=0,&\mbox{in}\;\Omega,
\end{cases}
\end{equation}
only has the trivial solution $(0,0)$ in $\mathbb{R}\times \bar{H}_0^{2}(\Omega)\times H_n^2(\Omega)$. From the second equation of \eqref{2022112901*} and the definition of ${H}_0^{2}(\Omega)$, it follows that $\psi\equiv0$. Integrating the third equation of \eqref{2022112901*}, one obtains
$$
\int_\Omega\frac{\alpha F'(c)}{d(c)}w_c\phi\mathrm{d}x=0,
$$
which together with $(H_2)$, $(H_3)$, $\phi\in \mathbb{R}$ and $w_c=\frac{c\alpha d(c)}{\theta|\Omega|}\int_\Omega(m-c)\mathrm{d}x>0$, implies
$
\phi=0.
$
This combined with the first and third equations of \eqref{2022112901*} gives that
$
\eta\equiv0.
$
So, the claim holds.

From the above claim and the implicit function theorem, it follows that there exists a neighborhood $\mathcal{U}_1\in\mathbb{R}\times \bar{H}_0^{2}(\Omega)\times H_n^2(\Omega)$ containing $(c,0,w_c)$ and a function $(\xi_\gamma,\zeta_\gamma,w_\gamma)$ defined for all $\gamma$ close to zero such that if $(\xi,\zeta,w)\in\mathcal{U}_1$ is a solution  of $\mathcal{L}_1(\xi,\zeta,w,\gamma)=0$ for some $\gamma$ close to zero, then we must have that $(u,w)=(\xi_\gamma+\zeta_\gamma,w_\gamma)$ is a positive solution of \eqref{KO2}. This along with  \eqref{2022090304} shows that \eqref{KO2} admits a unique positive solution when $\epsilon$ is large, and hence completes the proof.
\end{proof}

\subsection{Large/small predator diffusion}In this section, we shall investigate the uniqueness and asymptotic profile of solutions to \eqref{KO2} as $\mu \to \infty$ and $\mu \to 0$. First we define
$$g(u)=\displaystyle \frac{\int_\Omega{F(u)}{d^{-1}(u)}\mathrm{d}x}{\int_\Omega{d^{-1}(u)}\mathrm{d}x},$$
where $u\in C(\Omega;[0,+\infty))$. On top of assumptions $(H_2)$ and $(H_3)$, we impose two additional assumptions:
\begin{enumerate}
\item[($H_4$)]\quad $g'(u)>0$ for any $u\in C(\Omega;[0,+\infty))$, where $g'(u)$ denotes the Frechet derivative.
\item[($H_5$)]\quad $\left(\frac{F(u)}{ud(u)}\right)'\geq0$ for any $u\geq0$, where $'$ denotes the differentiation with respect to $u$.
\end{enumerate}
We give some examples where $(H_4)$ or $(H_5)$ holds. If $F(u)$ satisfies $(H_2)$ and $d(u)=e^{-ku}$ (or $(1+u)^{-k}$) for any $k\geq0$, then $(H_4)$ holds, see Lemma \ref{lem01*} and Proposition \ref{prop01*} for the proof.
If $F(u)=u$ and $d(u)$ satisfies $(H_3)$, or $F(u)=\frac{u}{1+u}$ (Holling-II) and $d(u)=(1+u)^{-k}$ (or $e^{-ku}$) with  $k\geq1$, then $(H_5)$ holds.
\begin{theorem}\label{thm4}
Assume $(H_1)$, $(H_2)$, $(H_3)$ and $(H_4)$ hold. Then the following results hold true.
\begin{itemize}
\item[(i)]
{If $\theta>\alpha g(\tilde{u})$, then system \eqref{KO2} doesn't have any positive solution when $\mu$ is large};
\item[(ii)]
{If $0<\theta<\alpha g(\tilde{u})$ and $(H_5)$ holds, then any positive solution of system \eqref{KO2} will converge to $(u^*,c^*)$ in $C^1(\Omega)$ as $\mu\to\infty$, where $c^*$ is a positive constant and  $(u^*,c^*)$ is the unique positive solution of \begin{equation}\label{2022083103}
\begin{cases}
\epsilon\Delta u^*+u^*(m(x)-u^*)-\frac{F(u^*)}{d(u^*)}c^*=0, &\mbox{in}\;\Omega,\\
\nabla u^* \cdot \n=0,&\mbox{on}\;\partial\Omega,\\
\int_\Omega\frac{\alpha F(u^*)-\theta}{d(u^*)}\mathrm{d}x=0.
\end{cases}
\end{equation}
Moreover, if $F(u)=u$, $d(u)=e^{-ku}$ (or $(1+u)^{-k}$) with $k\in[0,\frac{\alpha}{\theta}]$, then system \eqref{KO2} admits a unique positive solution when $\mu$  is large}.
\end{itemize}
\end{theorem}
\begin{proof}
For assertion (i), suppose by contradiction that system \eqref{KO2} admits a positive solution $(u_i,w_i)$ with $\mu=\mu_i$, where $\mu_i\to+\infty$ as $i\to\infty$. Then $\mu_i$ and $(u_i,w_i)$ satisfy
\begin{equation}\label{2022083001}
\begin{cases}
\epsilon\Delta u_i+u_i(m(x)-u_i)-\frac{F(u_i)}{d(u_i)}w_i=0, &\mbox{in}\;\Omega,\\
\mu_i\Delta w_i+\frac{\alpha F(u_i)-\theta}{d(u_i)}w_i=0, &\mbox{in}\;\Omega,\\
\nabla u_i \cdot \n=\nabla w_i \cdot \n=0,&\mbox{on}\;\partial\Omega.
\end{cases}
\end{equation}
Using Lemma \ref{lem07}, for any $i\geq 1$, one has
$$
\|u_i\|_{L^{\infty}(\Omega)}\leq m_{\max}\;\;
\mbox{and}\;\;
\|w_i\|_{L^{\infty}(\Omega)}\leq c_0.
$$
Similar to the analysis as that in the proof of  Theorem \ref{thm5},  one can deduce from (\ref{2022083001}) that $(u_i,w_i)$, passing to a subsequence if necessary, converges to some nonnegative function $(u_\infty,w_\infty)$ in  $C^1(\Omega)$ as $i\to\infty$,
where $(u_\infty,w_\infty)$ satisfies (in the weak sense)
\begin{equation}\label{2022083002}
\begin{cases}
\epsilon\Delta u_\infty+u_\infty(m(x)-u_\infty)-\frac{F(u_\infty)}{d(u_\infty)}w_\infty=0, &\mbox{in}\;\Omega,\\
\Delta w_\infty=0, &\mbox{in}\;\Omega,\\
\nabla u_\infty \cdot \n=\nabla w_\infty \cdot \n=0,&\mbox{on}\;\partial\Omega.
\end{cases}
\end{equation}
Therefore, there exists some constant $c_0\geq0$ such that $w_\infty=c_0$. We will show that $w_\infty=c_0\geq0$ can't occur.

If $w_\infty=0$, from the first equation of system \eqref{2022083002}, it follows that $u_\infty=0$ or $u_\infty=\tilde{u}$. We first show that $u_\infty=0$ can't occur. If not, assume $u_\infty=0$. Let $\hat{u}_i=\frac{u_i}{\|u_i\|_{L^\infty}}$. Then $\hat{u}_i$ satisfies
\begin{equation}\label{2022083003}
\begin{cases}
\epsilon\Delta \tilde{u}_i+\tilde{u}_i\left(m(x)-u_i-\frac{F(u_i)w_i}{u_id(u_i)}\right)=0, &\mbox{in}\;\Omega,\\
\nabla \tilde{u}_i \cdot \n=0,&\mbox{on}\;\partial\Omega.
\end{cases}
\end{equation}
By assumption $(H_2)$, one has $\lim\limits_{u\to0}\frac{F(u)}{u}=F'(0)$. Similar to the above analysis, one can deduce that $\hat{u}_i\to \hat{u}_\infty\geq0$ in $C^1(\Omega)$ as $i\to\infty$ (passing to a subsequence if necessary) and $\hat{u}_\infty$ satisfies $\|\hat{u}_\infty\|_{L^\infty}=1$.
Multiplying the first equation of \eqref{2022083003} by $\frac{1}{\hat{u}_i}$, integrating the resulting equation on $\Omega$ and letting $i\to\infty$, one gets
$$
\int_\Omega m(x)\mathrm{d}x=-\epsilon \int_\Omega\frac{|\nabla \hat{u}_\infty|^2}{\hat{u}_\infty}\mathrm{d}x\leq0,
$$
which contradicts the assumption $(H_1)$. On the other hand, if $u_\infty=\tilde{u}$, let $\tilde{w}_i=\frac{w_i}{\|w_i\|_{L^\infty}}$. Then $\tilde{w}_i$ satisfies
\begin{equation}\label{2022083004}
\begin{cases}
\mu_i\Delta \tilde{w}_i+\frac{\alpha F(u_i)-\theta}{d(u_i)}\tilde{w}_i=0, &\mbox{in}\;\Omega,\\
\nabla \tilde{w}_i \cdot \n=0,&\mbox{on}\;\partial\Omega.
\end{cases}
\end{equation}
Similarly, one can derive that $\tilde{w}_i\to1$ in $C^1(\Omega)$ as $i\to\infty$ which is equivalent to $\theta>\alpha g(\tilde{u})$. Integrating the first equation of \eqref{2022083004} on $\Omega$ and letting $i\to\infty$, one obtains
$$
\int_\Omega\frac{\alpha F(\tilde{u})-\theta}{d(\tilde{u})}\mathrm{d}x=0,
$$
which contradicts $\int_\Omega\frac{\alpha F(\tilde{u})-\theta}{d(\tilde{u})}\mathrm{d}x<0$. Therefore, $w_\infty=0$ can't occur.

If $w_\infty=C>0$, integrating the second equation of system  \eqref{2022083001} on $\Omega$ and letting $i\to\infty$, one has $\int_\Omega\frac{\alpha F(u_\infty)-\theta}{d(u_\infty)}\mathrm{d}x=0$. This indicates that
\begin{equation}\label{2022083102}
\theta=\alpha g(u_\infty).
\end{equation}
Combining systems \eqref{2022062101} and \eqref{2022083001} alongside the method of upper-lower solutions, we have
$$
u_\infty<\tilde{u}\quad\mbox{on}\;\;\Omega,
$$
which combined with $(H_4)$ gives that $g(u_\infty)<g(\tilde{u}).$   This together with \eqref{2022083102} implies that
$$
\theta<\alpha g(\tilde{u}),
$$
which contradicts our assumption $\theta>\alpha g(\tilde{u})$ and the assertion in statement (i) is proved.

Next we show the results stated in statement (ii). From lemma \ref{lem01} (ii) and Lemma \ref{lem02} (ii), it follows that $(\tilde{u},0)$ is linearly unstable for any $\mu>0$, which combined with Lemma \ref{lem09} suggests that system \eqref{KO2} admits at least one positive solution  for any $\mu>0$. We next establish the following claim.

{\it Claim 1: any positive solution of system \eqref{KO2}, denoted by $(u_\mu,w_\mu)$, converges to $(u^*,c^*)$ in $C^1(\Omega)$ as $\mu\to\infty$, where $c^*$ is a positive constant and  $(u^*,c^*)$ is the unique positive solution of \eqref{2022083103}.}
Similar to the argument as that in proving statement (i), it suffices to show that system \eqref{2022083103} admits a unique positive solution $(u^*,c^*)$ with $c^*$ being a positive constant. To this end, we introduce an auxiliary question
\begin{equation}\label{2022083104}
\begin{cases}
\epsilon\Delta z+z\Big(m(x)-z-\frac{F(z)}{zd(z)}c\Big)=0, &\mbox{in}\;\Omega,\\
\nabla z \cdot \n=0,&\mbox{on}\;\partial\Omega.
\end{cases}
\end{equation}
Since $\lim\limits_{z\to0}\frac{F(z)}{z}=F'(0)$, using the assumption $(H_5)$, it is standard to show (cf. \cite{CantrellCosner}) that for any $c\geq0$, \eqref{2022083104} admits a unique positive solution denoted by $z_c$. By the method of upper-lower solutions, one has that
\begin{equation}\label{2022083105}
\mbox{if $c_1>c_2\geq0$, then $z_{c_1}<z_{c_2}$ on $\bar{\Omega}$}.
\end{equation}
Clearly, we have
\begin{equation}\label{2022083106}
\mbox{$z_c=\tilde{u}$ when $c=0$.}
\end{equation}
Integrating the first equation of \eqref{2022083104} on $\Omega$ for any $c\geq0$, one obtains
$$
\int_\Omega z_c\left(m(x)-z_c-\frac{F(z_c)}{z_cd(z_c)}c\right)\mathrm{d}x=0,
$$
which together with \eqref{2022083105} yields that
\begin{equation}\label{2022083107}
z_c\to0\;\;\mbox{in}\;\;C(\bar{\Omega})\;\;\mbox{as}\;\;c\to+\infty.
\end{equation}
By the assumption $\theta<\alpha g(\tilde{u})$, we have from  \eqref{2022083106} and \eqref{2022083107} that
\begin{equation}\label{2022083108}
\lim\limits_{c\to+\infty}\alpha g(z_c)<\theta<\alpha g(\tilde{u}).
\end{equation}
This combined with the fact that $g(z_c)$ depends continuously and monotonically on $c$ shows that system \eqref{2022083104} admits a unique positive solution $u^*$ with $c=c^*$ such that
$
\theta=\alpha g(u^*).
$
Hence, Claim 1 holds.

Finally, we prove the second part of statement (ii). From now on, we assume $F(u)=u$, and $d(u)=e^{-ku}$ with $k\in[0,\frac{\alpha}{\theta}]$. The other case $d(u)=(1+u)^{-k}$ can be treated similarly.
Define $\mathcal{L}: H_n^{2}(\Omega)\times\mathbb{R}\times \bar{H}_0^{2}(\Omega)\times [0,+\infty)\to L^2(\Omega)\times\mathbb{R}\times \bar{L}^2(\Omega)$  by
\[
\mathcal{L}(u,\xi,\zeta,\beta)=
\begin{pmatrix}
\epsilon\Delta u+u(m(x)-u)-(\xi+\zeta)ue^{ku}
\\[6pt]	
\frac{1}{|\Omega|}\int_\Omega\left[(\alpha u-\theta)(\xi+\zeta)e^{ku}\right]\mathrm{d}x\\
\Delta \zeta+\beta[(\alpha u-\theta)(\xi+\zeta)e^{ku}-\frac{1}{|\Omega|}\int_\Omega (\alpha u-\theta)(\xi+\zeta)e^{ku}\mathrm{d}x]
\end{pmatrix}.
\]
Then, we have
$$
\begin{aligned}
&D_{(u,\xi,\zeta)}\mathcal{L}|_{(u,\xi,\zeta,\beta)=(u^*,c^*,0,0)}(\phi,\psi,\eta)\\
&\quad\quad=\begin{pmatrix}
\epsilon\Delta \phi+\phi(m(x)-2u^*)-c^*(1+ku^*)\phi e^{ku^*}-u^*\psi e^{ku^*}-u^*e^{ku^*}\eta\\[6pt]	
\frac{1}{|\Omega|}\int_\Omega\left[(\alpha+k\alpha u^*-k\theta)c^*\phi+(\alpha u^*-\theta)\psi+(\alpha u^*-\theta)\eta\right]e^{ku^*}\mathrm{d}x\\
\Delta \eta
\end{pmatrix},
\end{aligned}
$$
where $(u^*,c^*)$ is the unique positive solution of \eqref{2022083103}.

{\it Claim 2: $D_{(u,\xi,\zeta)}\mathcal{L}|_{(u,\xi,\zeta,\beta)=(u^*,c^*,0,0)}$ is non-degenerate}.
It suffices to show that problem
\begin{equation}\label{2022112901}
\begin{cases}
\epsilon\Delta \phi+\phi(m(x)-2u^*)-c^*(1+ku^*)\phi e^{ku^*}-u^*\psi e^{ku^*}-u^*e^{ku^*}\eta=0,&\mbox{in}\;\Omega,\\
\frac{1}{|\Omega|}\int_\Omega\left[(\alpha+k\alpha u^*-k\theta)c^*\phi+(\alpha u^*-\theta)\psi+(\alpha u^*-\theta)\eta\right]e^{ku^*}\mathrm{d}x=0\\
\Delta \eta=0,&\mbox{in}\;\Omega,
\end{cases}
\end{equation}
only admits trivial solution in $H_n^2(\Omega)\times\mathbb{R}\times \bar{H}_0^{2}(\Omega)$. The third equation of \eqref{2022112901} and the definition of ${H}_0^{2}(\Omega)$ suggest that $\eta\equiv0$. This along with the fact $\int_\Omega(\alpha u^*-\theta)e^{ku^*}\mathrm{d}x$, and the second equation of \eqref{2022112901} gives that
\begin{equation}\label{2022112902}
\int_\Omega(\alpha+k\alpha u^*-k\theta)\phi e^{ku^*}\mathrm{d}x=0.
\end{equation}
From the first equation of \eqref{2022083103} and the Krein-Rutman Theorem (cf. \cite{du2006order, KreinRutman1948}), one finds that
$$
\lambda_1(\epsilon,m-u^*-c^*e^{ku^*})=0
$$
which together with Lemma \ref{lem01} yields that
$$
\lambda_1(\epsilon,m-2u^*-c^*(1+ku^*)e^{ku^*})<0.
$$
This combined with the first equation of \eqref{2022112901} further implies that $\phi<0$ (resp. $>0$) on $\overline{\Omega}$ if $\psi>0$ (resp. $<0$) on $\overline{\Omega}$ with $\phi\equiv0$ if $\psi\equiv0$ on $\overline{\Omega}$.
This together with \eqref{2022112902} shows that $\phi\equiv\psi\equiv0$. So, Claim 2 holds.

Based on the Claim 2, $\mathcal{L}(u^*,c^*,0,0)=0$, and the implicit function theorem implies that there exists a neighborhood $\mathcal{U}\in H_n^2(\Omega)\times\mathbb{R}\times \bar{H}_0^{2}(\Omega)$ containing $(u^*,c^*,0)$ and a function $(u_\beta,\xi_\beta,\zeta_\beta)$ defined for all $\beta$ close to zero such that if $(u,\xi,\zeta)\in\mathcal{U}$ is a solution  of $\mathcal{L}(u,\xi,\zeta,\beta)=0$ for some $\beta$ close to zero, then we must have that $(u,w)=(u_\beta,\xi_\beta+\zeta_\beta)$ is a positive solution of \eqref{KO2}. This together with Claim 1 shows that \eqref{KO2} admits a unique positive solution when $\mu$  is large, which completes the proof.
\end{proof}

\begin{theorem}\label{thm6}
Suppose that $(H_1)$, $(H_2)$ and $(H_3)$ hold. Fixing all the parameters except $\mu$, assume that $0<\theta<\alpha F(\tilde{u}_{\max})$. Then every positive solution $(u_\mu,w_\mu)$ of system \eqref{KO2} satisfies  that  $u_\mu\to u_0>0$  uniformly on $\bar{\Omega}$ and $w_\mu\to w_0\geq0$ in $L^p(\Omega)$ for any $p\geq1$ as $\mu\to0$, where $u_0\leq F^{-1}(\theta/\alpha)$ on $\bar{\Omega}$ and satisfies (in the weak sense)
\begin{equation}\label{2022092002}
\begin{cases}
\epsilon\Delta u_0+u_0(m(x)-u_0)-\frac{F(u_0)}{d(u_0)}w_0=0, &\mbox{in}\;\Omega,\\
\nabla u_0 \cdot \n=0,&\mbox{on}\;\partial\Omega,
\end{cases}
\end{equation}
and
\begin{equation}\label{2022092101}
\mbox{$w_0(x)=0$ a.e. in $\{x\in\Omega|u_0(x)< F^{-1}(\theta/\alpha)\}$ and $|\{x\in\Omega|w_0>0\}|>0$.}
\end{equation}
Moreover, the following uniqueness results hold.
\begin{itemize}
\item[(a)]  If $m_{\min}\geq F^{-1}(\theta/\alpha)$, then the solution of \eqref{2022092002} is unique and given by
\begin{equation}\label{2022092102}
u_0\equiv F^{-1}(\theta/\alpha)\;\mbox{and}\;
w_0(x)=\frac{\alpha}{\theta}\cdot d\left(F^{-1}(\theta/\alpha)\right)F^{-1}(\theta/\alpha)\left[m(x)- F^{-1}(\theta/\alpha)\right]\;\mbox{a.e. in}\;\Omega.
\end{equation}
\item[(b)] If $m_{\min}<F^{-1}(\theta/\alpha)$ and $\Omega=(0,L)$, we have the following result:
if $m_x\geq0$ (resp. $m_x\leq0$)  in $(0,L)$, then there exists unique $y^*\in(0,L)$, where $y^*$ may be different for the cases $m_x\geq 0$ and $m_x\leq 0$, such that
\begin{equation}\label{2022092107}
     u_0(x)=
     \begin{cases}
     F^{-1}(\theta/\alpha),&\mbox{if}\;x\in[y^*,L]\ (\mathrm{resp}. \ x \in [0, y^*]),\\
   \tilde{u}_{y^*},&\mbox{if}\;x\in[0,y^*)\ (\mathrm{resp}. \ x \in [y^*, L]),
     \end{cases}
          \end{equation}
and
\begin{equation}\label{2022092108}
       w_0(x)=
     \begin{cases}
    \frac{\alpha}{\theta}\cdot d\left(F^{-1}(\theta/\alpha)\right)F^{-1}(\theta/\alpha)\left[m(x)- F^{-1}(\theta/\alpha)\right],&\mbox{a.e. in}\;(y^*,L) \ (\mathrm{resp}.\ in \ [0, y^*]),\\
   0,&\mbox{a.e. in}\;(0,y^*)\ (\mathrm{resp}.\ in \ [y^*, L]),
     \end{cases}
     \end{equation}
     where $\tilde{u}_{y^*}$ is the unique positive solution of
\begin{equation}\label{2022092106}
     \begin{cases}
\epsilon u_{xx}+u(m(x)-u)=0, &\mbox{in}\;(0,y^*),\\
u_x(0)=u_x(y^*)=0 \ (\mathrm{resp}. \ u_x(y^*)=u_x(L)=0),\quad u(y^*)=F^{-1}(\theta/\alpha).
\end{cases}
\end{equation}
\end{itemize}

\end{theorem}
\begin{proof}
From Lemma \ref{lem04} and Lemma \ref{lem09}, it follows that system \eqref{KO2} admits at least one positive solution denoted by $(u_\mu,w_\mu)$ when $\mu$ is small. By Lemma \ref{lem07}, one obtains that
\begin{equation}\label{2022092001}
\mbox{$(u_\mu,w_\mu)\to (u_0,w_0)$ in $L^p(\Omega)$ for any $p\geq1$,  as $\mu\to0$.}
\end{equation}
Clearly, $u_0\geq0$ and $w_0\geq0$ on $\bar{\Omega}$.
Moreover, applying the elliptic regularity (cf. \cite{GilbargTrudinger2001}) and the Sobolev imbedding theorem, we may assume that $u_\mu\to u_0$ in $C^1(\bar{\Omega})$ as $\mu\to0$ and $ (u_0,w_0)$ satisfies \eqref{2022092002}. Using the strong maximum principle to \eqref{2022092002}, we obtain that $u_0>0$ on $\bar{\Omega}$ or $u_0\equiv0$. If $u_0\equiv0$, similar to the proofs as those for Theorem \ref{thm4}, one can deduce that $\lambda_1(\epsilon,m)=0$, which is impossible due to assumption $(H_1)$. So, $u_0>0$ on $\bar{\Omega}$.
The equation for $w_\mu$ suggests that
$$
\lambda_1\left(\mu,\frac{\alpha F(u_\mu)-\theta}{d(u_\mu)}\right)=0,\quad\mbox{for any small}\;\;\mu>0,
$$
which combined with Lemma \ref{lem01} (ii) implies that
$$
0=\lim\limits_{\mu\to0}\lambda_1\left(\mu,\frac{\alpha F(u_\mu)-\theta}{d(u_\mu)}\right)=\max\limits_{x\in \bar{\Omega}}\frac{\alpha F(u_0)-\theta}{d(u_0)}.
$$
This together with assumptions $(H_2)$ and $(H_3)$ yields that
\begin{equation}\label{2022092003}
u_0\leq F^{-1}(\theta/\alpha) \;\mbox{on} \;\bar{\Omega}\;\mbox{and}\;u_0(x)=F^{-1}(\theta/\alpha) \;\mbox{for some $x\in\bar{\Omega}$}.
\end{equation}
Integrating the equation that $w_\mu$ satisfies on $\Omega$, one has
$$
\int_\Omega \frac{\alpha F(u_\mu)-\theta}{d(u_\mu)}w_\mu\mathrm{d}x=0,\quad\mbox{for any small}\;\;\mu>0.
$$
Sending $\mu\to0$, we have
$
\int_\Omega \frac{\alpha F(u_0)-\theta}{d(u_0)}w_0\mathrm{d}x=0,
$
which shows that
\begin{equation}\label{2022092004}
\mbox{$w_0=0$ a.e. in $\{x\in\Omega|u_0(x)< F^{-1}(\theta/\alpha)\}$.}
\end{equation}
We proceed to prove that $|\{x\in\Omega|w_0>0\}|>0$. If not, assume $w_0=0$ a.e. in $\Omega$. Combining \eqref{2022092002} and $u_0>0$, one has $u_0=\tilde{u}$ on $\bar{\Omega}$. Then, by the assumption $0<\theta<\alpha F(\tilde{u}_{\max})$ and $(H_2)$, we have
$$
\max\limits_{x\in \bar{\Omega}}u_0=\tilde{u}_{\max}>F^{-1}(\theta/\alpha),
$$
which contradicts \eqref{2022092003}. Therefore, $|\{x\in\Omega|w_0>0\}|>0$.

Next we consider the case $m_{\min}\geq F^{-1}(\theta/\alpha)$.
It is easy to verify that $(u_0,w_0)$ given in \eqref{2022092102} is well-defined and it satisfies \eqref{2022092002}, \eqref{2022092101} and $0<u_0\leq F^{-1}(\theta/\alpha)$ on $\bar{\Omega}$. So, it suffices to show that
\eqref{2022092002} admits a unique non-negative solution $(u_0,w_0)$ which satisfies \eqref{2022092101} and $0<u_0\leq F^{-1}(\theta/\alpha)$ on $\bar{\Omega}$. If not, assume that \eqref{2022092002} admits a non-negative solution $(u_1,w_1)\neq (u_0,w_0)$ (see \eqref{2022092102}) satisfying \eqref{2022092101} and $0<u_1\leq F^{-1}(\theta/\alpha)$ on $\bar{\Omega}$.
Then, $u_1\leq,\not\equiv F^{-1}(\theta/\alpha)$ in $\Omega$ and there exists some $x_0\in\Omega$ such that $u_1(x_0)<F^{-1}(\theta/\alpha)$. By continuity, one can find some neighborhood $x_0\in\Omega_1\subset\Omega$ such that
$$
u_1<F^{-1}(\theta/\alpha)\;\mbox{in}\;\Omega_1\quad\mbox{and}\quad u_1(x)=F^{-1}(\theta/\alpha),\;\mbox{for}\;x\in\partial\Omega_1\cap\Omega,
$$
which together with $u_1\leq,\not\equiv F^{-1}(\theta/\alpha)$ in $\Omega$  further implies that
$
\nabla u_1(x)=0,\;\mbox{for}\;x\in\partial\Omega_1\cap\Omega.
$
By \eqref{2022092101}, one has $w_1=0$ a.e. in $\Omega_1$. Therefore, $u_1$ satisfies
$$
\begin{cases}
\epsilon\Delta u+u(m(x)-u)=0, &\mbox{in}\;\Omega_1,\\
\nabla u\cdot \n=0,&\mbox{on}\;\partial\Omega_1.
\end{cases}
$$
However the maximum principle applied to the above equations yields that
$$
m_{\min}\leq\min\limits_{x\in\Omega_1}m(x)\leq\min\limits_{x\in\Omega_1} u_1(x)<F^{-1}(\theta/\alpha),
$$
which contradicts our assumption $m_{\min}\geq F^{-1}(\theta/\alpha)$. These facts complete the proof of the first part.

Finally, we consider the scenario $m_{\min}< F^{-1}(\theta/\alpha)$ and $\Omega=(0,L)$.  We shall only prove the case $m_x\geq 0$ and the case $m_x\leq 0$ can be shown similarly. For any $y\in(0,L]$, we consider an auxiliary problem
\begin{equation}\label{2022092105}
 \begin{cases}
\epsilon u_{xx}+u(m(x)-u)=0, &\mbox{in}\;(0,y),\\
u_x(0)=u_x(y)=0.
\end{cases}
\end{equation}
Without loss of generality,  we assume $m(0)>0$ (if $m(0)<0$, the method is still valid).
It is well-known that \eqref{2022092105} admits a unique positive solution denoted by $\tilde{u}_y$ (see Proposition \ref{prop2}). Moreover, if $m_x\geq,\not\equiv0$ in $(0,y)$, then $\frac{d\tilde{u}_y}{dx}>0$ in $(0,y)$ (see the proof of Lemma \ref{lem06*}); while if $m_x\equiv0$ in $(0,y)$, then $u\equiv m(0)$ in $(0,y)$.

{\bf Claim A}: if $0<y_1<y_2\leq L$, then $\tilde{u}_{y_2}(y_2)\geq\tilde{u}_{y_1}(y_1)$, where ``='' holds if and only if $m_x\equiv0$ in $(0,y_2)$. If $m_x\equiv0$ in $(0,y_2)$, then $\tilde{u}_{y_2}(y_2)=\tilde{u}_{y_1}(y_1)=m(0)$. If $m_x\geq,\not\equiv0$ in $(0,y_2)$, then $\frac{d\tilde{u}_{y_2}}{dx}>0$ in $(0,y_2)$ by Lemma \ref{lem06*}. Thus, $\tilde{u}_{y_2}$ restricted in $(0,y_1)$ is a strictly upper-solution of \eqref{2022092105} with $y=y_1$ due to $\frac{d\tilde{u}_{y_2}(y_1)}{dx}>0=\frac{d\tilde{u}_{y_1}(y_1)}{dx}$. Moreover let $\phi_1^*(\epsilon,m)>0$ be the principal eigenfunction of the following eigenvalue problem
$$
\begin{cases}
\epsilon \phi_{xx}+m\phi=\lambda\phi, &\mbox{in}\;(0,y_1),\\
\phi_x(0)=\phi_x(y_1)=0.
\end{cases}
$$
Then one can choose sufficiently small enough $\sigma>0$ such that $\sigma\phi_1^*(\epsilon,m)<\tilde{u}_{y_2}$ in $(0,y_1)$ and $\sigma\phi_1^*(\epsilon,m)$ is a strictly lower-solution of \eqref{2022092105} with $y=y_1$. Therefore, from the methods of upper-lower solution, it follows that
$$
\tilde{u}_{y_2}(y_1)>\tilde{u}_{y_1}(y_1)
$$
which together with $\frac{d\tilde{u}_{y_2}}{dx}>0$ in $(0,y_2)$ implies that
$$
\tilde{u}_{y_2}(y_2)>\tilde{u}_{y_1}(y_1).
$$
Hence Claim A is proved. On the other hand, one observes that
$$
\lim\limits_{y\to0}\tilde{u}_{y}(y)=m(0)=m_{\min}<F^{-1}\left(\frac{\theta}{\alpha}\right)\quad\mbox{and}\quad\lim\limits_{y\to L}\tilde{u}_{y}(y)=\tilde{u}(L)=\tilde{u}_{\max}>F^{-1}\left(\frac{\theta}{\alpha}\right),
$$
which together with Claim A implies that there exists unique $y^*$ such that $\tilde{u}_{y^*}$ satisfies \eqref{2022092106}. By $\frac{d\tilde{u}_{y^*}}{dx}(y^*)=0$ and $\frac{d\tilde{u}_{y^*}}{dx}(x)>0$ in $(0,y^*)$, one has $\frac{d^2\tilde{u}_{y^*}}{dx^2}(y^*)\leq0$, which substituted into the first equation of \eqref{2022092106} further gives that
$$
m(y^*)\geq \tilde{u}_{y^*}(y^*)=F^{-1}\left(\frac{\theta}{\alpha}\right).
$$
Therefore, $(u_0,w_0)$ defined in \eqref{2022092107} and \eqref{2022092108} satisfies \eqref{2022092002}, \eqref{2022092101}, and $0<u_0\leq F^{-1}\left(\frac{\theta}{\alpha}\right)$. To complete the proof, it suffices to show that \eqref{2022092002} admits a unique non-negative solution $(u_0,w_0)$ which satisfies \eqref{2022092101} and $0<u_0\leq F^{-1}\left(\frac{\theta}{\alpha}\right)$ on $[0,L]$. Assume that \eqref{2022092002} admits another non-negative solution $(u_2,w_2)$ which satisfies \eqref{2022092101} and $0<u_2\leq F^{-1}\left(\frac{\theta}{\alpha}\right)$ on $[0,L]$.

{\bf Claim B}: $u_2=u_0$ on $[0,L]$, where $u_0$ is given in \eqref{2022092107}.
We first prove that  $u_2=u_0$ on $[0,y^*]$. We note here that $u_0=\tilde{u}_{y^*}$ on $[0,y^*]$. It suffices to consider two cases
$$
(1)\;\;u_2(0)<F^{-1}\left(\frac{\theta}{\alpha}\right)\;\;\mbox{and}\;\; (2)\;\;u_2(0)=F^{-1}\left(\frac{\theta}{\alpha}\right).
$$

For case (1), it suffices to consider two cases
$$
\text{(1a)}\;\mbox{$\exists$ $x_1\in(0,L)$ such that $u_2(x_1)=F^{-1}\left(\frac{\theta}{\alpha}\right)$}\;, \ \mbox{and}\;\mathrm{(1b)}\;u_2<F^{-1}\left(\frac{\theta}{\alpha}\right)\;\mbox{in}\;(0,L).
$$
For case (1a), we define $x_2=\inf\limits_{x\in[0,L]}\{u_2(x)=F^{-1}\left(\frac{\theta}{\alpha}\right)\}$. Then, we have
$$
u_2<F^{-1}\left(\frac{\theta}{\alpha}\right)\;\;\mbox{on}\;[0,x_2)\;\;\mbox{and}\;\;u_2(x_2)=F^{-1}\left(\frac{\theta}{\alpha}\right).
$$
By \eqref{2022092101}, we have $w_2=0$ a.e. in $(0,x_2)$ and  $u_2$ satisfies
$$
  \begin{cases}
\epsilon u_{xx}+u(m(x)-u)=0, &\mbox{in}\;(0,x_2),\\
u_x(0)=u_x(x_2)=0,\quad u(x_2)=F^{-1}\left(\frac{\theta}{\alpha}\right).
\end{cases}
$$
Claim A shows that  $x_2=y^*$ and $u_2=\tilde{u}_{y^*}$ on $[0,y^*]$. For case (1b), by \eqref{2022092002} and \eqref{2022092101}, one has that $w_2=0$ a.e. in $(0,L)$ and $u_2=\tilde{u}$, which is impossible due to the fact that $\tilde{u}(L)>F^{-1}\left(\frac{\theta}{\alpha}\right)$.

For case (2), it suffices to consider two cases
$$
\text{(2a)}\;u_2\equiv F^{-1}\left(\frac{\theta}{\alpha}\right)\;\mbox{on}\;[0,L]\;,\ \mbox{and}\;\text{(2b)}\;\mbox{$\exists$ $x_3\in(0,L)$ such that $u_2(x_3)<F^{-1}\left(\frac{\theta}{\alpha}\right)$}.
$$
For case (2a), \eqref{2022092002} tells us that
$w_0(0)<0$ due to the assumption $m(0)=m_{\min}<F^{-1}\left(\frac{\theta}{\alpha}\right)$. This is impossible. For case (2b), we define
$$x_4=\sup\limits_{x\in[0,x_3]}\left\{u_2(x)=F^{-1}\left(\frac{\theta}{\alpha}\right)\right\}$$ and $$x_5=\inf\limits_{x\in[x_3,L]}\left\{u_2(x)=F^{-1}\left(\frac{\theta}{\alpha}\right)\right\}.$$ We note here that $x_4\in(0,x_3)$ and $x_5\in (x_3,L]$. Similarly, one obtains that $w_2=0$ a.e. in $(x_4,x_5)$ and  $u_2$ satisfies
\begin{equation}\label{2022092110}
  \begin{cases}
\epsilon u_{xx}+u(m(x)-u)=0, &\mbox{in}\;(x_4,x_5),\\
u_x(x_4)=u_x(x_5)=0,\quad u(x_4)=F^{-1}\left(\frac{\theta}{\alpha}\right)
\end{cases}
\end{equation}
which admits a unique solution which is non-decreasing with respect to $x$. This contradicts the fact
$u_2<F^{-1}\left(\frac{\theta}{\alpha}\right)\mbox{in}\;(x_4,x_5)$ by the definition of $x_4$ and $x_5$. Hence, $u_2=\tilde{u}_{y^*}=u_0$ on $[0,y^*]$.

Finally, it remains to show that $u_2\equiv F^{-1}\left(\frac{\theta}{\alpha}\right)$ on $[y^*,L]$. Recall that $u_2(y^*)=F^{-1}\left(\frac{\theta}{\alpha}\right)$. Arguing by contradiction, we assume that $\exists$ $x_6\in(0,L)$ such that $u_2(x_6)<F^{-1}\left(\frac{\theta}{\alpha}\right)$.  Define
$$x_7=\sup\limits_{x\in[y^*,x_6]}\left\{u_2(x)=F^{-1}\left(\frac{\theta}{\alpha}\right)\right\}$$ and $$x_8=\inf\limits_{x\in[x_6,L]}\left\{u_2(x)=F^{-1}\left(\frac{\theta}{\alpha}\right)\right\}.$$
Then, similar to the analysis in case (2b), one can deduce a contradiction. Therefore, $u_2\equiv F^{-1}\left(\frac{\theta}{\alpha}\right)$
on $[y^*,L]$. This completes the proof of Claim B and hence the proof of Theorem \ref{thm6}.
\end{proof}

\bigbreak

\noindent \textbf{Acknowledgement}.
The research of D. Tang was supported by the National Natural Science Foundation of China (No. 11901596), Science and Technology Program of Guangzhou (No. 202102020772), and the Fundamental Research Funds for the Central Universities (No. 2021qntd20).
Z.-A. Wang is partially supported by a grant from the NSFC/RGC Joint Research Scheme sponsored by the Research Grants Council of Hong Kong and the National Natural Science Foundation of China (Project No. \mbox{$\mathrm{N}_{-}$PolyU509/22}).


\end{document}